\documentclass[12pt]{article}
\usepackage{enumerate,amsfonts,dsfont,amsmath,accents,amsthm,mathrsfs,cases,bbm,amssymb,color}
\usepackage{indentfirst} %首段首行缩进
\usepackage[bookmarks]{hyperref}      %书签
\usepackage{hypernat}
\usepackage{fancyhdr}
\usepackage{esint}

\allowdisplaybreaks      %多行公式自动分页
\newtheorem{theorem}{Theorem}[section]
\newtheorem{lemma}[theorem]{Lemma}
\newtheorem{proposition}[theorem]{Proposition}
\newtheorem{corollary}[theorem]{Corollary}
\theoremstyle{definition}
\newtheorem{definition}[theorem]{Definition}

\newtheorem{remark}[theorem]{Remark}
\numberwithin{equation}{section}

\newcommand{\cart}{\mathrm{cart\,}}

\def\geq{\geqslant}
\def\leq{\leqslant}
\numberwithin{equation}{section}
\begin{document}
\title{Cartesian currents in fractional Sobolev spaces and applications to functions of bounded higher variation}
\author{Qiang Tu\and Chuanxi Wu\and
Faculty of Mathematics and Statistics, Hubei University, Wuhan 430062, China\footnote{\emph{Email addresses}:~qiangtu@whu.edu.cn(Qiang Tu),
cxwu@hubu.edu.cn (Chuanxi Wu).}.
}
\date{}   %不显示日期
\maketitle
\noindent{\bf Abstract:} In this paper we establish weak continuity results for the distributional Jacobian minors in fractional sobolev spaces, which can be seen as a extension of
recent work of Brezis and Nguyen on the distributional Jacobian determinant.
Then we apply  results to introducing the currents associated with graphs of maps in  fractional Sobolev spaces and study some  relevant properties such as structure properties, weak continuity and so on.
As another application, we extend the definition of functions of bounded higher variation, which defined by Jerrard and Soner in $W^{1,N-1}\cap L^{\infty}(\Omega,\mathbb{R}^N)$ ( $\mbox{dim}\Omega\geq N$), to $W^{1-\frac{1}{N},N}(\Omega,\mathbb{R}^N)$  and obtain some meaningful  results including weak coarea formula, strong coarea formula and chain rule.
\medskip

\medskip

\noindent{\bf 2010 MR Subject Classification:}  46E35, 46F10,  49Q20, 49Q15.
%\medskip

\section{Introduction and main results}
This paper is devoted to the studying of some  properties and applications of the distributional minors of Jacobian matrix  of  non-smooth  functions defined from $\Omega$, a bounded Lipschitz domain of $\mathbb{R}^n$,  into $\mathbb{R}^N$( $n,N\geq 2$).

  Starting with the seminal work of Morrey\cite{MC}, Reshetnyak\cite{RY} and Ball\cite{BJ}, it is well known that the distributional Jacobian determinant
 $\mbox{Det}(Du)$ of a map $u\in W^{1,\frac{n^2}{n+1}}(\Omega,\mathbb{R}^n)$ (or $u\in L^{q}\cap W^{1,p}(\Omega,\mathbb{R}^n)$ with $\frac{n-1}{p}+\frac{1}{q}=1$ and $n-1\leq p\leq \infty$) is defined by
 $$\mbox{Det}(Du):=\sum_{j} \partial_j(u^i(\mbox{adj} Du)^i_j),$$
 where $\mbox{adj}Du$ means the adjoint matrix of $Du$.
 Compared with the classical cases about distributional Jacobian determinant, Brezis-Nguyen \cite{BN} discussed the range of the map $u\mapsto \mbox{Det} (Du)$ in the framework of fractional Sobolev spaces. They showed that the distributional Jacobian determinant $\mbox{Det}(Du)$ for any $u\in W^{1-\frac{1}{n},n}(\Omega,\mathbb{R}^n)$ can be defined as
 $$\langle\mbox{Det}(Du), \psi\rangle:=\lim_{k\rightarrow \infty}\int_{\Omega}\det(Du_k)\psi dx~~~\forall \psi \in C_{c}^{1}(\Omega, \mathbb{R})$$
where $u_k\in C^1(\overline{\Omega}, \mathbb{R}^n)$ such that $u_k\rightarrow u$ in $ W^{1-\frac{1}{n},n}$. Furthermore, they pointed that the result is optimal in the framework of the space $W^{s,p}$, i.e., the distributional Jacobian determinant is well-defined in $W^{s,p}$ if and only if $W^{s,p}\subseteq W^{1-\frac{1}{n},n}$.

We recall that for $0<s<1$ and $1\leq p<\infty$, the fractional Sobolev space $W^{s,p}(\Omega)$ is defined by
$$W^{s,p}(\Omega):=\left\{u\in L^p(\Omega)\mid \left(\int_{\Omega}\int_{\Omega} \frac{|u(x)-u(y)|^p}{|x-y|^{n+sp}}dxdy\right)^{\frac{1}{p}}<\infty\right\},$$
and the norm
$$\|u\|_{W^{s,p}}:=\|u\|_{L^p}+\left(\int_{\Omega}\int_{\Omega} \frac{|u(x)-u(y)|^p}{|x-y|^{n+sp}}dxdy\right)^{\frac{1}{p}}.$$
A natural problem is raised whether we can extend the definitions and results for  the distributional Jacobian determinant to the distributional Jacobian minors in fractional Sobolev spaces. Our first result gives a positive answer to the question.

\begin{theorem}\label{subdetthm001}
Let $p$ be integer with $2\leq p\leq \underline{n}:=\min\{n,N\}$ and $\alpha\in I(k,n),\beta\in I(k,N)$ with $0\leq k\leq p$, the Jacobian minor operator $u \longmapsto M_{\alpha}^{\beta}(Du) (\mbox{see}~ (\ref{sub2001})):C^1(\Omega,\mathbb{R}^N)\rightarrow \mathcal{D}'(\Omega)$ can be extended uniquely as a continuous mapping $u \longmapsto \mbox{Div}_{\alpha}^{\beta}(Du):W^{1-\frac{1}{p},p}(\Omega,\mathbb{R}^N)\rightarrow \mathcal{D}'(\Omega)$. Moreover
for all $u,v\in W^{1-\frac{1}{p},p}(\Omega,\mathbb{R}^N)$, $\psi \in  C^1_c(\Omega,\mathbb{R})$ and $1\leq k\leq p$, we have
$$\left|\langle\mbox{Div}_{\alpha}^{\beta}(Du)-\mbox{Div}_{\alpha}^{\beta}(Dv),\psi\rangle\right|\leq C_{k,p,n,N,\Omega}\|u-v\|_{W^{1-\frac{1}{p},p}}\left(\|u\|_{W^{1-\frac{1}{p},p}}^{k-1} +\|v\|_{W^{1-\frac{1}{p},p}}^{k-1}\right)\|D\psi\|_{L^{\infty}}.$$
In particular, the distribution minor $\mbox{Div}_{\alpha}^{\beta}(Du)$ can be expressed as
\begin{equation}\label{sub1002}
\langle \mbox{Div}_{\alpha}^{\beta}(Du), \psi \rangle=-\sum_{i\in \alpha+(n+1)}\sigma(\alpha+(n+1)-i,i)\int_{\Omega\times (0,1)}M^{\beta}_{\alpha+(n+1)-i}(DU)\partial_i\Psi d\widetilde{x}
\end{equation}
for any extension $U\in W^{1,p}(\Omega\times (0,+\infty), \mathbb{R}^N)$ and $\Psi\in C_c^1(\Omega\times [0,1),\mathbb{R})$ of $u$ and
$\psi$, respectively.
\end{theorem}

 Next we will restrict our attention to the Cartesian  currents associated to the graph of a $W^{s,p}$-map $u$ from $\Omega$ to $\mathbb{R}^N$. The initial motivation of our work is the following:
 Giaquinta-Modica-Sou\v{c}ek \cite{GMS} introduced  a class of functions $u\in L^1(\Omega, \mathbb{R}^N)$, named $\mathcal{A}^1(\Omega, \mathbb{R}^N)$,  such that $u$ is approximately differentiable a.e. and all minors of the Jacobian matrix $Du$ are summable in $\Omega$.
For $u\in \mathcal{A}^1(\Omega,\mathbb{R}^N)$, one can define an integer multiplicity (i.m.)
rectifiable current $G_{u}$ carried by the  rectifiable graph of $u$. More precisely, $G_u$ is defined for $\omega\in \mathcal{D}^n(\Omega\times \mathbb{R}^N)$ by
\begin{equation}\label{sub1001}
G_u(\omega)= \sum_{|\alpha|+|\beta|=n}  \sigma(\alpha,\overline{\alpha}) \int_{\Omega} \omega_{\alpha\beta}(x,u(x)) M_{\overline{\alpha}}^{\beta} (Du(x)) dx
\end{equation}
where $\omega=\sum_{|\alpha|+|\beta|=n}\omega_{\alpha\beta}(x,y) dx^{\alpha}\wedge dy^{\beta}$, for more details see \cite[Vol. I, Sect. 3.2.1]{GMS} or Sec. 2.

It is easy to see  that the current $G_u$ is determined by all  Jacobian minors of $u$. Now the goal of this thesis is to develop the Cartesian currents theory by  distributional Jacobian minors of non-smooth functions in  fractional Sobolev spaces. In particular, our result is inspired by a recent result of Acerbi and Mucci \cite{AM} characterizing currents associated with graphs of maps in trace spaces $W^{1-\frac{1}{p},p}(\mathcal{X},\mathcal{Y})$ that have vanishing mean oscillation,   where $\mathcal{X}$ and $\mathcal{Y}$ are smooth, connected, compact Riemannian manifolds without boundary. They introduced the notion of
  semi-current $G_u$ carried by the graph of a map $u$ in a  trace space $W^{1-\frac{1}{p},p}(\mathcal{X},\mathcal{Y})$, and extended the semi-current $G_u$ to a  current $T_u\in \mathcal{D}_n(\mathcal{X}\times\mathcal{Y})$, actually an integral flat chain, for maps $W^{1-\frac{1}{p},p}\cap VMO(\mathcal{X},\mathcal{Y})$ with $p\geq \mbox{dim}\mathcal{Y}$.

We now give a formal statement of our main results. We refer to Sec. 2, 4 below for the notation.
\begin{theorem}\label{subdetthm002}
 Let $u\in W^{1-\frac{1}{n'},n'}(\Omega,\mathbb{R}^N)$ and $U\in W^{1,n'}(\Omega\times (0,1), \mathbb{R}^N)$ be any extension of $u$, where $n':=\min\{n+1,N\}$. Then the boundary current
$$T_u:=(-1)^{n-1}(\partial G_U)\llcorner \Omega\times \mathbb{R}^N$$
is well-defined as an $n$-dimension current in $\mathcal{D}_n(\Omega\times \mathbb{R}^N)$. Moreover the following properties hold:
\begin{enumerate}
\item[{\em(\romannumeral1)}]  $\partial T_u \llcorner \Omega\times \mathbb{R}^N=0$.
\item [{\em(\romannumeral2)}] For any  $\alpha\in I(n-k,n)$, $\beta\in I(k,N)$ with $0\leq k\leq \underline{n}$, $\psi\in C^{\infty}_c(\Omega\times \mathbb{R}^N)$ and  $\Psi\in C^{\infty}_c(\Omega\times [0,1)\times \mathbb{R}^N)$ with $\Psi|_{\Omega\times \mathbb{R}^N}=\psi$, we have
\begin{equation}
\begin{split}
(T_u)^{\alpha\beta}(\psi)=-\sum_{i\in \overline{\alpha}+(n+1)} \sigma(\overline{\alpha}+(n+1)-i,i)\sigma(\alpha,\overline{\alpha})\int_{\Omega\times (0,1)}
D_{x_i} [\Psi(\widetilde{x},U(\widetilde{x}))] M^{\beta}_{\overline{\alpha}+(n+1)-i}(DU)d\widetilde{x}
\end{split}
\end{equation}
\item [{\em(\romannumeral3)}]If $u\in W^{1,n'}(\Omega,\mathbb{R}^N)$, the current $T_u$ is  consistent with the current $G_u$ defined in the class $\mathcal{A}^1(\Omega,\mathbb{R}^N)$.
\item [{\em(\romannumeral4)}]$\mbox{Div}_{\overline{\alpha}}^{\beta}(Du)=\sigma(\alpha,\overline{\alpha})\pi_{\sharp}T_u^{\alpha\beta}$
for any  $\alpha\in I(n-k,n)$, $\beta\in I(k,N)$ with $0\leq k\leq \underline{n}$.
\item [{\em(\romannumeral5)}]
 If $\{u_j\}_{j=1}^{\infty}, u\subset W^{1-\frac{1}{n'},n'}(\Omega,\mathbb{R}^N)$ with $u_j\rightarrow u$ in $W^{1-\frac{1}{n'},n'}(\Omega,\mathbb{R}^N)$. Then
$$T_{u_j}\rightharpoonup T_u~~~~\mbox{in}~D_n(\Omega\times \mathbb{R}^N).$$
\end{enumerate}
\end{theorem}
Note that although the current $T_u$ associated to the graph of such a $W^{1-\frac{1}{n'},n'}$-map is determined by all distributional  Jacobian minors of $u$,
in general it may have infinite mass, see Example 3.9 in \cite{AM}. Furthermore, under the hypotheses of finite mass, we give the following theorem

\begin{theorem}\label{subdetthm003}
Let  $u\in W^{1-\frac{1}{n'},n'}(\Omega,\mathbb{R}^N)$. If $T_u$ has finite mass, then $T_u\in \cart(\Omega\times \mathbb{R}^N)$(\mbox{see}~(\ref{subcart01})),  $u\in BV(\Omega, \mathbb{R}^N)\cap \mathcal{A}^1(\Omega, \mathbb{R}^N)$ and
\begin{equation}\label{subdetformula0031}
\mbox{Div}_{\overline{\alpha}}^{\beta}(Du)^{\mbox{ac}}(x)=M_{\overline{\alpha}}^{\beta}( ap Du(x)), ~~\mbox{for}~ \mathcal{L}^n \mbox{a.e.}~x\in \Omega,
\end{equation}
for any $\alpha \in I(n-k,n)$, $\beta\in I(k,N)$ with $1\leq k\leq \underline{n}$, where $\mbox{Div}_{\overline{\alpha}}^{\beta}(Du)^{\mbox{ac}}$ is the absolutely continuous part of $\mbox{Div}_{\overline{\alpha}}^{\beta}(Du)$ with respect to the Lebesgue measure and
$ap Du$ is the approximate differential of $u$.
\end{theorem}
Moreover we introduce the  notion of semi-current and study some properties  in the fractional Sobolev space $W^{1-\frac{1}{p},p}(\Omega,\mathbb{R}^N)$ with $2\leq p\leq n'$. For more general results and details see Sec. 4.

\medskip
Finally we pay attention to the functions of bounded higher variation which is introduced by Jerrard-Soner \cite{JS}. Given a function $u\in W^{1,N-1}\cap L^{\infty}(\Omega,\mathbb{R}^N)$ with $2\leq N\leq n$, the distributional Jacobian $[Ju]$ of $u$ is defined by
\begin{equation}\label{sub1003}
[Ju]=\sum_{\alpha\in I(N,n)}\mbox{Det}(u_{x_{\alpha_1}},\cdot\cdot\cdot,u_{x_{\alpha_N}})de_{\alpha},
\end{equation}
where $I(N,n)$ refer to (\ref{subnotation01}).  A function $u\in W^{1,N-1}\cap L^{\infty}(\Omega,\mathbb{R}^N)$ has
bounded $N$-variation in $\Omega$ if
$$\sup \left\{\langle[Ju],\omega\rangle:=\sum_{\alpha\in I(N,n)}\int_{\Omega} \omega^{\alpha} \mbox{Det}(u_{x_{\alpha_1}}dx,\cdot\cdot\cdot,u_{x_{\alpha_N}})\mid \omega\in D^{N}(\Omega),~ \|\omega\|\leq 1 \right\}<\infty.$$
We write $BNV(\Omega, \mathbb{R}^N)$ to denote the space of functions of bounded $N$-variation.
The definition of $BNV$ more or less generalizes that of the classical space $BV$. Jerrard-Soner extended some results about $BV$, such as coarea formula and chain rule, to the space $BNV$. In particular, they show a structure theorem in $BNV(\Omega, S^{N-1})$ which is an analogue of the theorem of De Giorgi on the rectifiability of the reduced boundary of a set of finite perimeter. A similar structure result has been given by Hang-Lin \cite{HL} for $u\in W^{1-\frac{1}{N},N}(\Omega, S^{N-1})$. Hang-Lin pointed out that although their result has lower order requirement on differentiability, Jerrard-Soner's result is not contained in theirs in view of the embedding theorem. Furthermore, Lellis \cite{LE1,LE2} proved a strong coarea-type formula and a chain rule for $[Ju]$ and  applied  some properties of currents on metric spaces developed by Ambrosio-Kirchheim \cite{AK} to studying the structure of $[Ju]$ in the space $BNV$.

According to Theorem \ref{subdetthm001}, it follows that distributional Jacobian minors is well-defined in fractional Sobolev spaces. A natural problem is  whether we can extend the definitions and results for $BNV$ in $W^{1,N-1}\cap L^{\infty}(\Omega,\mathbb{R}^N)$  to fractional Sobolev spaces. Here we try to discuss this problem. More precisely, we define a class of functions of bounded $N$-variation in $W^{1-\frac{1}{N},N}(\Omega,\mathbb{R}^N)$,$(2\leq N\leq n)$, and apply Cartesian currents theory to show some  results, such as  coarea formula, chain rule and structure theorem, which more or less generalizes the works of the classical $BNV$ in the framework of $W^{1,N-1}\cap L^{\infty}(\Omega,\mathbb{R}^N)$.
 The following is our weak coarea  formula. We refer to Sec. 5 below for the notation

\begin{theorem}[\bf Weak coarea formula]\label{subth502}
Let $u\in W^{1-\frac{1}{N},N}(\Omega, \mathbb{R}^N)$  with $2\leq N\leq n$. Then for $\mathcal{H}^N$-a.e. $y\in \mathbb{R}^N$,
there exist a current $[J,u,y]\in \mathcal{D}_{n-N}(\Omega)$  and an integer multiplicity $T_yu\in \mathcal{R}_{n-N+1}(\Omega \times (0,1))$  such that
$$[Ju]=\int_{\mathbb{R}^N}  [J,u,y]dy$$
in the sense of currents and
$$[J,u,y]=(\partial T_yu)\llcorner \Omega.$$
 I.e. for any  any $\omega\in \mathcal{D}^{n-N}(\Omega)$ and $\widetilde{\omega}\in \mathcal{D}^{n-N}(\Omega\times [0,1))$ with $\widetilde{\omega}|_{\Omega\times \{0\}}=\omega$, then
$$[Ju](\omega)=\int_{\mathbb{R}^N} T_yu(d\widetilde{\omega})dy.$$
More precisely, for any extension $U\in W^{1,N}(\Omega\times (0,1))$ of $u$, $T_yu$ can be written as
$$T_yu=\tau (U^{-1}(y)\cap E_{U},1, \zeta),$$
where $E_U:=R_U\cap \{\widetilde{x}\in \Omega\times(0,1)\mid J_{U}(\widetilde{x})>0\}$ and the  orientation $\zeta$ can be split as
\begin{equation}\label{subth5021}
\zeta=(-1)^{(n-1)}\frac{\sum_{\alpha\in I(n-N+1,n+1)}\sigma(\alpha,\overline{\alpha})M_{\overline{\alpha}}^{\overline{0}}(DU) e_{\alpha}}{|\sum_{\alpha\in I(n-N+1,n+1)}\sigma(\alpha,\overline{\alpha})M_{\overline{\alpha}}^{\overline{0}}(DU) e_{\alpha}|}.
\end{equation}
\end{theorem}
\begin{remark}
Indeed $W^{1,N-1}\cap L^{\infty}(\Omega,\mathbb{R}^N)\subset W^{1-\frac{1}{N},N}(\Omega,\mathbb{R}^N)$, Theorem \ref{subth502} can be seen as a  generalization  of Hang-Lin and Jerrard-Soner's work \cite{HL,JS} for the weak coarea formula of $BNV$. It should be observed that the difinition of current $[Ju]$ in this article(see (\ref{subdef501})) slightly differs from the notation of (\ref{sub1003}).
\end{remark}
\begin{theorem}[\bf Chain rule]\label{subdetthm005}
Let $u\in W^{1-\frac{1}{N},N}(\Omega, \mathbb{R}^N)$  with $2\leq N\leq n$ and $F\in C^{1}(\mathbb{R}^N,\mathbb{R}^N)$ with $DF\in L^{\infty}$, then $F(u)\in W^{1-\frac{1}{N},N}(\Omega, \mathbb{R}^N)$ and
\begin{equation}\label{subcr001}
[JF(u)](\omega)=\int_{\mathbb{R}^N} \det DF(y) [J,u,y](\omega)dy
\end{equation}
for any $\omega\in D^{n-N}(\Omega)$.
\end{theorem}
Notice that Theorem \ref{subth502} and \ref{subdetthm005} do not require $u\in BNV$.

In general $BNV$ fails to inherit strong coarea formula of $BV$, see\cite[Theorem 1.4]{JS}, but we prove that if $u$ satisfies certain conditions then the strong coarea formula holds. The following is our  result.
\begin{theorem}\label{subdetthm006}
Let $u\in BNV(\Omega, \mathbb{R}^N)$. Moreover suppose that
\begin{enumerate}
\item[{\em(\romannumeral1)}]   $\int_{\mathbb{R}^N} \left\| [J,u,y]\right\|(\Omega)dy<\infty$.
\item [{\em(\romannumeral2)}] Either $u\in C^0(\Omega)$ or
$$\left\|[Ju]\right\|(V)=\liminf_{j\rightarrow \infty}\{\left\|[Ju_j]\right\|(V)\mid u_j\in C^{\infty}\cap BNV(\Omega,\mathbb{R}^N),~u_j\rightarrow u~\mbox{in}~W^{1-\frac{1}{N},N}\}$$
for any open set $V\subset \Omega$.
\end{enumerate}
Then
$$\left\|[Ju]\right\|(A)=\int_{\mathbb{R}^N} \left\| [J,u,y]\right\|(A)dy$$
for any Borel set $A\subset \Omega$. Moreover this remains true if $u\in W^{1,N}(\Omega, \mathbb{R}^N)$.
\end{theorem}
\medskip

This paper is organized as follows. Some facts and notions about  Cartesian currents are given in Section 2.
In Section 3 we  establish weak continuity result and  the definition for distribution Jacobian minors   in fractional Sobolev spaces.
Then we introduce Cartesian currents for  fractional Sobolev spaces and study some properties in Section 4.
Finally in Section 5 we give some  results, such as coarea formula and chain rule and so on, for functions of bounded higher variation in  $W^{1-\frac{1}{N},N}(\Omega, \mathbb{R}^N)$.

\section{Preliminaries}

\vskip .2cm
In this section we collect some notation and preliminary results. We refer to \cite{GMS,FH} for general facts about  Geometric Measure Theory, whereas further details concerning   Cartesian currents can be found in  \cite{GMS}.

\noindent{\bf Notation for multi-indices.}
Fix $0\leq k\leq n$, we shall use the standard notation for ordered multi-indices
\begin{equation}\label{subnotation01}
I(k,n):=\{\alpha=(\alpha_1,\cdot\cdot\cdot,\alpha_k) \mid \alpha_i  ~\mbox{integers}, 1\leq \alpha_1 <\cdot\cdot\cdot< \alpha_k\leq n\},
\end{equation}
where $n \geq 2$. Set $I(0,n)=\{0\}$ and $|\alpha|=k$ if $\alpha \in I(k,n)$.
For $\alpha\in I(k,n)$,
\begin{enumerate}
\item[{\em(\romannumeral1)}]   $\overline{\alpha}$ is the element in $I(n-k,n)$ which complements $\alpha$  in $\{1,2,\cdot\cdot\cdot,n\}$ in the natural increasing order.
\item [{\em(\romannumeral2)}] $\alpha-i$ means the multi-index of length $k-1$ obtained by removing $i$ from $\alpha$ for any  $i \in \alpha$.
\item [{\em(\romannumeral3)}] $\alpha+j$ means the multi-index of length $k+1$ obtained by adding j to $\alpha$ for  any $j\notin \alpha$, .
\item [{\em(\romannumeral4)}] $\sigma(\alpha,\beta)$ is the sign of the permutation which reorders $(\alpha,\beta)$ in the natural increasing order for any   multi-index  $\beta$ with $\alpha\cap \beta=\emptyset$. In particular set $\sigma(\overline{0},0):=1$.
\end{enumerate}
 Moreover we denote by $e_1,\cdot\cdot\cdot,e_n$ and $\varepsilon_1,\cdot\cdot\cdot,\varepsilon_N$ the canonical bases in $\mathbb{R}^n$ and $\mathbb{R}^N$, respectively. So we write $e_{\alpha}:=e_{\alpha_1}\wedge\cdot\cdot\cdot \wedge e_{\alpha_k}$ for any $\alpha \in I(k,n)$.

Let $n,N \geq 2$ and  $A=(a_{ij})_{N \times n}$ be an $N \times n$ matrix.
Given two ordered multi-indices $\alpha\in I (k,n)$ and $\beta \in I(k,N)$, then
$A_{\alpha}^{\beta}$ denotes
the $k \times k $-submatrix of $A$ obtained by selecting the rows and columns by $\beta$ and $\alpha$, respectively. Its determinant will be denoted by

\begin{equation}\label{sub2001}
M_{\alpha}^{\beta}(A):=\det A_{\alpha}^{\beta},
\end{equation}
and we set $M_{0}^{0}(A):=1$.
The adjoint of $A_{\alpha}^{\beta}$ is  defined  by the formula
$$(\mbox{adj}~ A_{\alpha}^{\beta})_j^i:= \sigma(i,\beta-i) \sigma(j,\alpha-j) \det A_{\alpha-j}^{\beta-i},~~~~ i \in \beta, j\in \alpha.$$
So Laplace  formulas can be written as
$$M_{\alpha}^{\beta}(A)= \sum_{j \in \alpha} a_{ij} (\mbox{adj}~ A_{\alpha}^{\beta})_j^i,~~~~ i\in\beta.$$

%\noindent{\bf Approximate differentiability.}
%Let $u: \Omega\rightarrow \mathbb{R}$ be a measurable function, the approximate upper limit of $u$ at $x\in\Omega$  is defined as the number (eventually $+\infty$ or $-\infty$) given by
%$$ap\limsup_{y\rightarrow x, y\in \Omega} u(y):=\inf\left\{t\mid \limsup_{r\rightarrow 0}\frac{\mathcal{L}^n(E_t^{+}\cap B(x,r))}{\mathcal{L}^n(B(x,r))}=0 \right\},$$
%where $B(x,r)$ is a ball with center $x$, radius $r$ and $E_t^{+}:=\{x'\in\Omega \mid u(x')>t\}$.
%$u$ is approximately differentiable at $x$ if there exists a linear mapping $L: \mathbb{R}^n \rightarrow \mathbb{R}$ such that  the approximate upper limit of the measurable function $\frac{|u(y)-u(x)-L(y-x)|}{|y-x|}$ at $x\in\Omega$  is equal to $0$, i.e.,
%$$ap\limsup_{y\rightarrow x,y\in\Omega}\frac{|u(y)-u(x)-L(y-x)|}{|y-x|}=0.$$
%For more general properties about approximate differentiability, see \cite[Vol. I, Sec. 3.1.4]{GMS}.

\noindent{\bf Currents and rectifiable currents.}
 We denote by $\mathcal{D}^k(\Omega)$($0\leq k \leq n$)the space of compactly supported $k$-form in $\Omega$.
The dual space to $\mathcal{D}^k(U)$ is the class of $k$-currents $\mathcal{D}_k(\Omega)$.
\begin{definition}
For any open set $V\subset\subset \Omega$, the mass of a current $T\in \mathcal{D}_k(\Omega)$ in $V$ is defined by
$$\mathbf{M}_{V}(T):=\sup\{T(\omega) \mid \omega \in \mathcal{D}^k(\Omega),~\mbox{spt}~ \omega \subset V, \| \omega\|\leq 1 \},$$
 and $\mathbf{M}(T):=\mathbf{M}_{\Omega}(T)$ denotes the mass of $T$.  We set
 $$\mathcal{M}_k(\Omega):=\{T\in \mathcal{D}_k(\Omega)\mid \mathbf{M}(T)<\infty\}.$$
 \end{definition}
 If $T\in \mathcal{M}_k(\Omega)$, by the dominated convergence theorem the action of $T$ extends to  all differential forms with bounded Borel coefficients in $\Omega$.
 \begin{definition}
 For any integer $k$ with $1\leq k\leq n$,  a set $M\subset \Omega$ is said to be countably k-rectifiable if
 $$M=M_0\cup (\cup_{j=1}^{\infty} F_j(A_j))$$
 where $\mathcal{H}^n(M_0)=0$ and $F_j: A_j\subset \mathbb{R}^k\rightarrow \mathbb{R}^{n}$ is Lipschitz.
  \end{definition}
A current $T\in \mathcal{D}_k(\Omega)$ is said to be of the type $\tau(\mathcal{M}, \theta, \xi)$ if
$$T(\omega)=\int_{\mathcal{M}}\langle \omega(x), \xi(x)\rangle \theta(x) d\mathcal{H}^k(x),~~~~\mbox{for any}~\omega \in \mathcal{D}^k(\Omega),$$
where $\mathcal{M}$ is an $\mathcal{H}^k$-measurable countably $k$-rectifiable subset of $\Omega$, the multiplicity $\theta: \mathcal{M}\rightarrow [0,\infty]$ is $\mathcal{H}^k$-measurable and locally $\mathcal{H}^k \llcorner \mathcal{M}$-summable, and $\xi:\mathcal{M}\rightarrow \bigwedge_k \mathbb{R}^n$ is
$\mathcal{H}^k$-measurable with $|\xi|=1$ for $\mathcal{H}^k\llcorner \mathcal{M}$-a.e.

\begin{definition}
A current $T=\tau(\mathcal{M}, \theta, \xi)$ is called an  integer multiplicity rectifiable $k$-current (briefly i.m. rectifiable $k$-current) if
$\theta$ is integer-valued
and  $\xi(x)$ provides an orientation to the approximate tangent spaces  $Tan^k(\mathcal{M}, x)$
for $\mathcal{H}^k$-a.e. $x\in\mathcal{M}$.
Here $\theta$ is called the multiplicity and $\xi$ is called the orientation for $T$.
Furthermore, the i.m. rectifiable  $k$-currents in  $\mathcal{D}_k(\Omega)$ is  denoted by $\mathcal{R}_k(\Omega)$ if $T$ has finite mass.
\end{definition}

Let  $T=\tau (\mathcal{M},\theta,\xi) \in \mathcal{R}_k(\Omega)$, and   $f: \Omega\rightarrow V \subset \mathbb{R}^n$  be a Lipschitz map such that  $f_{| \mbox{spt}~T} $  is proper. Then the push-forward of $T$ under $f$  turns out to be an i.m rectifiable  $k$-current  which can be explicitly written as (see \cite[vol. I, Sect. 2.2.4]{GMS} or \cite[4.1.28]{FH})
\begin{equation}\label{formula1}
\begin{split}
f_{\sharp} T(\omega) &= \int_{\mathcal{M}}\langle w(f(x)),(\bigwedge_k d^{\mathcal{M}}f)\xi(x)\rangle \theta(x) d\mathcal{H}^k(x)\\
                           &=\int_{f(\mathcal{M})}\langle\omega(y), \sum_{x\in f^{-1}(y) \cap \mathcal{M}_{+}} \theta (x) \frac{(\bigwedge_k d^{\mathcal{M}}f)\xi(x)}{|(\bigwedge_k d^{\mathcal{M}}f)\xi(x)|}\rangle d\mathcal{H}^k (y),\\
\end{split}
\end{equation}
where
$$\mathcal{M}_+=\{x \in \mathcal{M} \mid J_f^{\mathcal{M}}(x):=|(\bigwedge_k d^{\mathcal{M}}f)\xi(x)|>0\}.$$

\noindent{\bf Currents carried by graphs.}
\begin{definition}
The class of functions $\mathcal{A}^1(\Omega,\mathbb{R}^N)$ is defined by
\begin{equation}
\begin{split}
\mathcal{A}^1(\Omega,\mathbb{R}^N):=\{u\in L^1(\Omega,\mathbb{R}^N)&\mid u ~\mbox{is approximately defferentiable a.e.},\\
                           &~M_{\overline{\alpha}}^{\beta}(Du)\in L^1(\Omega)~\mbox{for any}~\alpha, \beta ~\mbox{with}~|\alpha|+|\beta|=n\}.\\
\end{split}
\end{equation}
\end{definition}
For $u \in \mathcal{A}^1(\Omega,\mathbb{R}^N)$,
 the i.m. rectifiable $n$-current $G_u \in \mathcal{D}_n(\Omega \times \mathbb{R}^N)$ of the type $\tau(\mathcal{G}_{u,\Omega},1,\xi_{u})$ is defined for $\omega \in \mathcal{D}^n(\Omega \times \mathbb{R}^N)$ by (see \cite[Vol. I, Sec. 3.2.1]{GMS})
\begin{align*}
  G_{u}(\omega) &= \int_{\mathcal{G}_{u,\Omega}} \langle\omega,\xi_{u}\rangle d \mathcal{H}^n
                   = \int_{\Omega} \langle\omega(x,u(x)),M(Du(x))\rangle dx \\
                   &= \sum_{|\alpha|+|\beta|=n}  \sigma(\alpha,\overline{\alpha}) \int_{\Omega} \omega_{\alpha\beta}(x,u(x)) M_{\overline{\alpha}}^{\beta} (Du(x)) dx,
\end{align*}
where the rectifiable graph
$$\mathcal{G}_{u,\Omega}:=\{(x,u(x))\mid x\in \mathcal{L}_u\cap A_D(u)\cap \Omega\},$$
here $\mathcal{L}_u$ is the set of Lebesgue points, $A_D(u)$ is the set of approximate differentiability points of $u$. Moreover
 $M(Du(x))$ is the $n$-vector in $\bigwedge_n(\mathbb{R}^{n+N})$ given by
$$M(Du(x))=(e_1+\sum_{i=1}^{N} D_1 u^i(x) \epsilon_i)\wedge ...\wedge(e_n+\sum_{i=1}^{N} D_n u^i(x) \epsilon_i)$$
and $\xi_u:=\frac{M(Du(x))}{|M(Du(x))|}$ is an unit $n$-vector which orients $\mathcal{G}_{u,\Omega}$.

Since $G_u$ has finite mass for $u \in \mathcal{A}^1(\Omega,\mathbb{R}^N)$, we may consider $G_u$ as a linear functional on $\mathcal{D}^n(\mathbb{R}^n\times \mathbb{R}^N)$, then the measure theoretic boundary of $G_u$ on $\mathbb{R}^n\times \mathbb{R}^N$ is defined by
$$\partial G_u(\omega):=G_u(d\omega),$$
for any $\omega\in \mathcal{D}^n(\mathbb{R}^n\times \mathbb{R}^N)$. The product structure in $\mathbb{R}_x^n\times \mathbb{R}_y^N$ induces a natural splitting of the exterior differential operator $d$ in $\mathbb{R}^n\times \mathbb{R}^N$ as
$$d=d_x+d_y$$
and $\partial G_u$ splits into its components $(\partial G_u)_{(k)}, 0\leq k\leq \underline{n}$, defined by
$$(\partial G_u)_{(k)}(\omega):=\partial G_u(\omega^{(k)}),$$
i.e., by testing $\partial G_u$ on the $n-1$-forms with exactly $k$ differentials with respect to $y$. Hence we can write
\begin{equation}
(\partial G_u)_{(k)}(\omega)= (G_{u})_{(k)}(d_x\omega)+(G_{u})_{(k+1)}(d_y\omega).
\end{equation}

In order  to characterize the Cartesian currents  $T$ which can be approximated by smooth graphs, such question connected with the problem of relaxation of the area integral for nonparametric graphs,   Giaquinta-Modica-Sou\v{c}ek \cite{GMS2} and \cite[Vol. II, Sec. 6]{GMS} introduced two classes of  Cartesian maps as
$$\mbox{cart}^1(\Omega,\mathbb{R}^N):=\{u\in \mathcal{A}^1(\Omega,\mathbb{R}^N)\mid \partial G_u \llcorner \Omega\times\mathbb{R}^N=0\};$$
\begin{equation}\label{subcart01}
\begin{split}
\mbox{cart}(\Omega,\mathbb{R}^N):=\{&T\in \mathcal{D}_n(\Omega,\mathbb{R}^N) \mid T~\mbox{is an i.m. rectifiable current in }~\Omega\times\mathbb{R}^N,\\
                           & \mathbf{M}(T)+\mathbf{M}(\partial T)<\infty,\|T\|_1<\infty, T^{\overline{0}0}\geq 0,
                            \pi_{\sharp}T=[\![\Omega]\!], \partial T \llcorner \Omega\times\mathbb{R}^N=0\}.
\end{split}
\end{equation}
It is clearly that $W^{1,\underline{n}}(\Omega,\mathbb{R}^N)\in \mbox{cart}^1(\Omega,\mathbb{R}^N)\cap \mbox{cart}(\Omega,\mathbb{R}^N)$.

\section{The distribution Jacobian minors}
In this section we extend the definition of distribution  minors of the Jacobian matrix $Du$ to fractional Sobolev spaces and  establish  the corresponding  weak continuity result.

Note that in the sequel we will let $p$ be a integer. We begin with the following simple lemma.

\begin{lemma}\label{subdetlem001}
Let $u\in C^1(\Omega,\mathbb{R}^N), \psi \in C^1_c(\Omega)$, $\underline{n}:= \min\{n,N\}\geq 2$ and $\alpha\in I(k,n), \beta\in I(k,N)$ with $0\leq k\leq \underline{n}$. Then
$$\int_{\Omega} M_{\alpha}^{\beta}(Du)\psi dx=-\sum_{i\in \alpha+(n+1)}\sigma(\alpha+(n+1)-i,i)\int_{\Omega\times (0,1)}M^{\beta}_{\alpha+(n+1)-i}(DU)\partial_i\Psi dxdx_{n+1},$$
for any extensions $U\in C^1(\Omega\times[0,1),\mathbb{R}^N)\cap C^2(\Omega\times (0,1),\mathbb{R}^n)$ and $\Psi\in C^1_c(\Omega\times[0,1),\mathbb{R})$ of $u$ and $\psi$, respectively.
\end{lemma}
\begin{proof}
It is easy to show the results in case $k=0,1$ and so we give the proof only for the case $2\leq k\leq \underline{n}$.
Denote $\partial_i:=\frac{\partial}{\partial x_i}$ and $d\widetilde{x}:= dxdx_{n+1}$. Applying the fundamental theorem of calculus, we have
\begin{equation}\label{sdformula301}
\begin{split}
\int_{\Omega} M_{\alpha}^{\beta}(Du)\psi dx&=-\int_{\Omega\times (0,1)} \partial_{n+1} \left(M_{\alpha}^{\beta}(DU)\Psi\right) d\widetilde{x}\\
&=-\int_{\Omega\times (0,1)} \partial_{n+1}M_{\alpha}^{\beta}(Du)\Psi d\widetilde{x}-\int_{\Omega\times (0,1)}M_{\alpha}^{\beta}(Du)\partial_{n+1}\Psi d\widetilde{x}.
\end{split}
\end{equation}
We denote the first part integral on the right-hand side by $I$, Laplace formulas imply that
\begin{align*}
I&=-\sum_{i\in \alpha} \sum_{j\in \beta}\int_{\Omega\times (0,1)} \sigma(i,\alpha-i)\sigma(j,\beta-j) \partial_{n+1}\partial_i U^j M_{\alpha-i}^{\beta-j}(DU) \Psi d\widetilde{x}\\
&=\sum_{i\in \alpha} \sum_{j\in \beta} \int_{\Omega\times (0,1)} \sigma(i,\alpha-i)\sigma(j,\beta-j) \partial_{n+1} U^j \left(\partial_i M_{\alpha-i}^{\beta-j}(DU) \Psi+M_{\alpha-i}^{\beta-j}(DU) \partial_i\Psi\right) d\widetilde{x}.
\end{align*}
Since
$$\sum_{i\in \alpha} \partial_i \left((\mbox{adj} (DU)^{\beta}_{\alpha})_i^j\right)=0$$
for any $j\in \beta$, it follow that
\begin{align*}
I&=\sum_{i\in \alpha}\sum_{j\in \beta} \int_{\Omega\times (0,1)} \sigma(i,\alpha-i)\sigma(j,\beta-j)  \partial_{n+1} U^j M^{\beta-j}_{\alpha-i}(DU) \partial_i\Psi d\widetilde{x}\\
&=\sum_{i\in \alpha}\sigma(i,\alpha-i)\sigma(n+1,\alpha-i)   \int_{\Omega\times (0,1)}  M_{\alpha+(n+1)-i}^{\beta}(DU) \partial_i\Psi d\widetilde{x}\\
&=-\sum_{i\in \alpha}\sigma(\alpha+(n+1)-i,i)\int_{\Omega\times (0,1)}  M_{\alpha+(n+1)-i}^{\beta}(DU) \partial_i\Psi d\widetilde{x}.
\end{align*}
Combing with the formula (\ref{sdformula301}), we obtain the desired conclusion.
\end{proof}

Using the above results and the trace theory, we can obtain an estimate similar to the Lemma 4 in \cite{BN}:

\begin{lemma}\label{subdetlem002}
 Let $u\in C^1(\overline{\Omega},\mathbb{R}^N), \psi \in C^1_c(\Omega)$,  $2\leq p\leq \underline{n}$ and $\alpha\in I(k,n), \beta\in I(k,N)$ with $1\leq k\leq p$. Then
\begin{align*}
\left|\int_{\Omega} M_{\alpha}^{\beta}(Du)\psi-M_{\alpha}^{\beta}(Dv)\psi dx \right|
\leq C_{k,p,n,N,\Omega}\|u-v\|_{W^{1-\frac{1}{p},p}}\left(\|u\|_{W^{1-\frac{1}{p},p}}^{k-1} +\|v\|_{W^{1-\frac{1}{p},p}}^{k-1}\right)\|D\psi\|_{L^{\infty}(\Omega)}
\end{align*}
\end{lemma}
\begin{proof}
Let $\widetilde{u}$ and $\widetilde{v}$ be extensions of $u$ and $v$ to $\mathbb{R}^n$ such that
$$\|\widetilde{u}\|_{W^{1-\frac{1}{p},p}(\mathbb{R}^n,\mathbb{R}^N)}\leq C_{n,N,p,\Omega} \|u\|_{W^{1-\frac{1}{p},p}(\Omega,\mathbb{R}^N)},~~~~\|\widetilde{v}\|_{W^{1-\frac{1}{p},p}(\mathbb{R}^n,\mathbb{R}^N)}\leq C_{n,N,p,\Omega}\|v\|_{W^{1-\frac{1}{p},p}(\Omega,\mathbb{R}^N)}$$
and
$$\|\widetilde{u}-\widetilde{v}\|_{W^{1-\frac{1}{p},p}(\mathbb{R}^n,\mathbb{R}^N)}\leq C_{n,N,p,\Omega} \|u-v\|_{W^{1-\frac{1}{p},p}(\Omega,\mathbb{R}^N)}.$$
According to  a well known trace theorem of Stein in \cite{STE1,STE2}, where $W^{1-\frac{1}{p},p}(\mathbb{R}^n)$ is identified as the space of traces of $W^{1,k}(\mathbb{R}^n\times(0,+\infty))$, there is a bounded linear extension operator
$$E:W^{1-\frac{1}{p},p}(\mathbb{R}^n,\mathbb{R}^N)\rightarrow W^{1,p}(\mathbb{R}^n\times (0,+\infty),\mathbb{R}^N).$$

Let $U$ and $V$  be extensions of $u$ and $v$ to $\mathbb{R}^n\times (0,+\infty)$, respectively, such that
$$U=Eu,~~V=Ev.$$
We then have
$$\|DU\|_{L^{p}(\Omega \times (0,1))}\leq C_{n,N,p,\Omega} \|u\|_{W^{1-\frac{1}{p},p}(\Omega,\mathbb{R}^N)},~~~~\|DV\|_{L^{p}(\Omega \times (0,1))}\leq C_{n,N,p,\Omega} \|v\|_{W^{1-\frac{1}{p},p}(\Omega,\mathbb{R}^N)}$$
and
$$\|DU-DV\|_{L^{p}(\Omega \times (0,1))}\leq C_{n,N,p,\Omega} \|u-v\|_{W^{1-\frac{1}{p},p}(\Omega,\mathbb{R}^N)}.$$
Let $\Psi \in C^1_c(\Omega\times [0,1))$ be an extension of $\psi$ such that
$$\|D\Psi\|_{L^{\infty}(\Omega\times [0,1))}\leq C_{n,N,\Omega}\|D\psi\|_{L^{\infty}(\Omega)}.$$
According to Lemma \ref{subdetlem001}, we have
\begin{align*}
&\left|\int_{\Omega} M_{\alpha}^{\beta}(Du)\psi-M_{\alpha}^{\beta}(Dv)\psi dx \right|\\
&\leq \sum_{i\in \alpha+(n+1)}\int_{\Omega\times (0,1)}|M^{\beta}_{\alpha+(n+1)-i}(DU)-M^{\beta}_{\alpha+(n+1)-i}(DV)|\cdot|\partial_i\Psi|d\widetilde{x}\\
&\leq C_{k} \int_{\Omega\times (0,1)}|DU-DV|(|DU|^{k-1}+|DV|^{k-1})|D\Psi| d\widetilde{x}\\
&\leq C_{k,\Omega} \|DU-DV\|_{L^{k}(\Omega \times (0,1))}(\|DU\|^{k-1}_{L^{k}(\Omega \times (0,1))}+\|DV\|^{k-1}_{L^{k}(\Omega \times (0,1))})\|D\Psi\|_{L^{\infty}(\Omega\times [0,1))}\\
&\leq C_{k,p,n,N,\Omega} \|u-v\|_{W^{1-\frac{1}{p},p}}\left(\|u\|_{W^{1-\frac{1}{p},p}}^{k-1} +\|v\|_{W^{1-\frac{1}{p},p}}^{k-1}\right)\|D\psi\|_{L^{\infty}}
\end{align*}
\end{proof}

Then we can give the  definition of  distributional minors  of $Du$ with the order less than $p$ when $u\in W^{1-\frac{1}{p},p}(\Omega,\mathbb{R}^N)$  with $2\leq p\leq \underline{n}$.
\begin{definition}\label{subdetdef001}
Let $u\in W^{1-\frac{1}{p},p}(\Omega, \mathbb{R}^N)$  with $2\leq p\leq \underline{n}$ . For any $\alpha\in I(k,n), \beta\in I(k,N)$ with $0\leq k\leq p$, the distributional minor of indices $\alpha$
and $\beta$ of $Du$, denoted by $\mbox{Div}_{\alpha}^{\beta}(Du)$, is defined by
\begin{equation}
\begin{split}
\langle \mbox{Div}_{\alpha}^{\beta}(Du), \psi \rangle:=\begin{cases}
\int_{\Omega} \psi(x)dx,~~~~~~~~~~~k=0; \\
\lim_{j\rightarrow \infty} \int_{\Omega} M^{\beta}_{\alpha}(Du_j)\psi dx,~~~~ 1\leq k\leq p\\
\end{cases}
\end{split}
\end{equation}
for any $\psi\in C^1_c(\Omega)$ and any sequence $\{u_j\}_{j=1}^{\infty}\subset C^1(\overline{\Omega},\mathbb{R}^N)$ such that $u_j\rightarrow u$ in $W^{1-\frac{1}{p},p}(\Omega,\mathbb{R}^N)$.
\end{definition}
\begin{remark}
This quantity is well-defined since $\mbox{Div}_{\alpha}^{\beta}(Du)$ is independent of the choice of the sequence and the fact that $C^1(\overline{\Omega})$ is dense in $W^{1-\frac{1}{p},p}(\Omega)$. Moreover  we denote $L^1:=W^{0,1}$, then the estimate in Lemma \ref{subdetlem002} also holds and hence  distributional minors  in case $p=1$ is well-defined  which can be reduced to well-known facts from the theory of BV-functions.
\end{remark}

%An argument similar to the one used in \cite[Prop. 3]{BN} shows  a fundamental representation of the distribution minor.
\begin{proof}[\bf Proof of Theorem 1.1]
The proof of the conclusion (\ref{sub1002}) is analogous to that in Proposition 3 in \cite{BN}  and the remainder of the argument is an immediate consequence following from the standard approximation argument, Definition \ref{subdetdef001}  and the estimate in Lemma \ref{subdetlem002}.
\end{proof}

In \cite{BN}, Brezis-Nguyen showed that the distributional Jacobian is well-defined in $W^{s,p}$ if and only if $W^{s,p}\subseteq W^{1-\frac{1}{N},N}$, the optimal result is based on an elaborate construction: a sum of well-chosen atoms, scaled at lacunary frequencies. It is no less reasonable to believe that $W^{1-\frac{1}{p},p}$ is the optimal space in the framework of fractional Sobolev space for the $p$ order distributional minor. Indeed,  by the embedding properties of fractional order Sobolev spaces, it is enough to construct the counter example in three cases, while the construction also can be refer to the  work of  Brezis-Nguyen in \cite{BN}.
%---------------------------------------------------------------------------------------%
%---------------------------------------------------------------------------------------%

%---------------------------------------------------------------------------------------%

\section{Cartesian currents in fractional Sobolev spaces}
In this section  we discuss the notion and some properties of semi-current $T^{p-1}_u$ in the fractional Sobolev space $W^{1-\frac{1}{p},p}(\Omega,\mathbb{R}^N)$ with $2\leq p\leq n':=\min\{n+1,N\}$.  In particular we show that the action of the semi-current $T_u^{n'-1}$ can be extended to a current $T_u$ in $\mathcal{D}_n(\Omega\times \mathbb{R}^N)$, actually an integral flat chain,  if $u\in W^{1-\frac{1}{n'},n'}(\Omega,\mathbb{R}^N)$.

For any $\omega\in \mathcal{D}^q(\Omega \times \mathbb{R}^N)$ with $1\leq q\leq n$, it can be written as
 $\omega=\sum_{i=0}^{\underline{q}} \omega^{j}$,
where $\underline{q}:=\min \{q,N\}$ and $ \omega^j$'s are the $q$-forms that contain exactly $j$ differentials in the vertical $\mathbb{R}^N_{y}$ variables. For any $0\leq r\leq \underline{q}$, we denote $\mathcal{D}^{q,r}(\Omega\times \mathbb{R}^N)$ the subspace of $\mathcal{D}^q(\Omega \times \mathbb{R}^N)$ of $q$-forms of the type $\omega=\sum_{i=0}^{r} \omega^{j}$.

The dual space of "semi-currents" is denoted by $\mathcal{D}_{q,r}(\Omega\times \mathbb{R}^N)$.
If $T\in \mathcal{D}_{q,r}(\Omega\times \mathbb{R}^N)$, we can splits it as a sum
 $$T=\sum_{|\alpha|+|\beta|=q}T^{\alpha\beta},$$
 where $\alpha\in I(q-k,n)$ and $\beta\in I(k,N)$ with $0\leq k\leq r$, $T^{\alpha\beta}\in \mathcal{D}'(\Omega\times \mathbb{R}^N)$ is defined by
$$T^{\alpha\beta}(\psi):=T(\psi(x,y)dx^{\alpha}\wedge dy^{\beta})~~~~\mbox{for any}~\psi\in C_{c}^{\infty}(\Omega\times \mathbb{R}^N).$$

We begin with the following simple lemma which is an importance property due to Giaquinta-Modica-Sou\v{c}ek \cite[Vol. I, Sec. 3.2.5]{GMS} for the boundaries and traces of Cartesian currents.
\begin{lemma}\label{sublem42}
Let $\Omega$ be a bounded Lipschitz domain in $\mathbb{R}^n$ and $U,V\in W^{1,p}(\Omega,\mathbb{R}^N)$ with $1\leq p\leq \underline{n}$. Then $U=V$ on $\partial \Omega$ implies
\begin{equation}\label{sublem421}
(\partial G_U)_{(k)}=(\partial G_V)_{(k)}
\end{equation}
for all $k$ with $0\leq k\leq p-1$.
\end{lemma}

According to the trace theory of fractional Sobolev spaces and the above Lemma, we have the following definition
\begin{definition}\label{subdef41}
Let $u\in W^{1-\frac{1}{p},p}(\Omega,\mathbb{R}^N)$  with $2\leq p\leq n':=\min\{n+1,N\}$, the $(n,p-1)$-current $T^{p-1}_u$ in $\mathcal{D}_{n,p-1}(\Omega \times \mathbb{R}^N)$ is given by
$$T^{p-1}_u:=(-1)^{n-1}(\partial G_U)\llcorner \Omega\times \mathbb{R}^N ~\mbox{on}~\mathcal{D}^{n,p-1}(\Omega \times \mathbb{R}^N),$$
where $U\in W^{1,p}(\Omega\times (0,1), \mathbb{R}^N)$ is any extension  of $u$,
I.e., for any $\omega \in \mathcal{D}^{n,p-1}(\Omega \times \mathbb{R}^N)$
$$T^{p-1}_u(\omega)=(-1)^{n-1} G_U (d\widetilde{\omega}),$$
for any extension $\widetilde{\omega}\in \mathcal{D}^{n,p-1}(\Omega\times [0,1)\times \mathbb{R}^N)$ of $\omega$.
\end{definition}
\begin{remark}
This quantity is well-defined since $T_u^{p-1}$ is independent of the choose of the extensions and $(\partial G_U)_k$ has support in $\partial(\Omega\times (0,1))\times \mathbb{R}^N$.
 Moreover the  definition cannot be extended to $\mathcal{D}^{n,p}(\Omega\times \mathbb{R}^N)$ since the equality (\ref{sublem421}) may not be true if $p<n'$. The construction of counter-examples can be seen in \cite[Vol. I, Sec. 3.2.5]{GMS}.
\end{remark}
Following the theory of Cartesian currents by Giaquinta-Modica-Sou\v{c}ek \cite{GMS},
we  show the following  properties for the  current $T_u^{p-1}$ carried by the graph of $u$.
\begin{proposition}\label{subpo32}
Let $u\in W^{1-\frac{1}{p},p}(\Omega,\mathbb{R}^N)$ and $T^{p-1}_u$ be given as above. Then
\begin{enumerate}
\item[{\em(\romannumeral1)}]  $\partial T^{p-1}_u (\xi) =0$ for any $\xi\in \mathcal{D}^{n-1,p-2}(\Omega\times \mathbb{R}^N)$ or $\xi\in \mathcal{Z}^{n-1,p-1}(\Omega\times \mathbb{R}^N):=\{\xi\in \mathcal{D}^{n-1,p-1}(\Omega\times \mathbb{R}^N)\mid d_y \xi=0\}$.
\item [{\em(\romannumeral2)}]  For any  $\alpha\in I(n-k,n)$, $\beta\in I(k,N)$ with $0\leq k\leq p-1$ and $\psi\in C^{\infty}_c(\Omega\times \mathbb{R}^N)$, we have
\begin{equation}\label{subfor43}
\begin{split}
(T^{p-1}_u)^{\alpha\beta}(\psi)=-\sum_{i\in \overline{\alpha}+(n+1)} \sigma(\overline{\alpha}+(n+1)-i,i)\sigma(\alpha,\overline{\alpha})\int_{\Omega\times (0,1)}
D_{x_i} [\Psi(\widetilde{x},U(\widetilde{x}))] M^{\beta}_{\overline{\alpha}+(n+1)-i}(DU)d\widetilde{x},
\end{split}
\end{equation}
where $U\in W^{1,p}(\Omega\times (0,1), \mathbb{R}^N)$ and $\Psi\in C_c^{\infty}(\Omega\times [0,1), \mathbb{R})$ be  extensions of $u$ and $\psi$, respectively.
In particular, $(T^{p-1}_u)^{\overline{0}0}(\psi(x,y))=\int_{\Omega}\psi(x,u(x))dx$.
\end{enumerate}
\end{proposition}
\begin{proof}
Then an argument in Remark 3 in \cite[Vol. I, Sec. 3.2.3]{GMS}
yields that
$$(\partial G_U)_{(k)} \llcorner \Omega\times(0,1)\times \mathbb{R}^N=0$$
for all $k\leq p-1$.
For any $\omega\in \mathcal{D}^{n,p-1}(\Omega\times \mathbb{R}^N)$, we can choose its extension $\widetilde{\omega}\in \mathcal{D}^{n,p-1}(\Omega\times[0,1)\times \mathbb{R}^N)$ as $\widetilde{\omega}=\omega\wedge \eta(x_{n+1})$
, where $\eta\in C^{\infty}([0,1),[0,1])$ is a  cut-off function satisfying
$$\eta(t)=1~ \mbox{for} ~0\leq t\leq \frac{1}{4},~~ \eta(t)=0~ \mbox{for}~ \frac{3}{4}\leq t< 1 ~~~\mbox{and}~ \|\eta'(t)\|\leq 4.$$
For any $\xi\in \mathcal{D}^{n-1,p-2}(\Omega\times \mathbb{R}^N)$ or $\xi\in \mathcal{Z}^{n-1,p-1}(\Omega\times \mathbb{R}^N)$, we have $d\xi\in \mathcal{D}^{n,p-1}(\Omega\times \mathbb{R}^N)$ and  $\xi\wedge d\eta\in \mathcal{D}^{n,p-1}(\Omega\times(0,1)\times \mathbb{R}^N)$, it implies that
$$\partial T^{p-1}_u (\xi)=(-1)^{n-1} G_U (d(\widetilde{d\xi}))=(-1)^{n-1} G_U (d(d\xi\wedge\eta))=- G_U (d(\xi\wedge d\eta))=0.$$
Obviously the $n+1$-form $d(\Psi(\widetilde{x},y)dx^{\alpha}\wedge dy^{\beta})$ can be written as
$$
d(\Psi(\widetilde{x},y)dx^{\alpha}\wedge dy^{\beta})=\sum_{i\in \overline{\alpha}+(n+1)}\sigma(i,\alpha)\partial_{x_i}\Psi dx^{\alpha+i}\wedge dy^{\beta}+\sum_{j\in \overline{\beta}}
(-1)^{|\alpha|}\sigma(j,\beta) \partial_{y_j} \Psi dx^{\alpha}\wedge dy^{\beta+j}.
$$
Hence
\begin{align*}
(T_u^{p-1})^{\alpha\beta}(\psi)&=(-1)^{n-1} G_{U}\left(d(\Psi(\widetilde{x},y)dx^{\alpha}\wedge dy^{\beta})\right)\\
&=(-1)^{n-1} \sum_{i\in \overline{\alpha}+(n+1)}\sigma(i,\alpha) \sigma(\alpha+i,\overline{\alpha}+(n+1)-i) \int_{\Omega\times (0,1)} \partial_{x_i}\Psi (\widetilde{x},U(\widetilde{x}))M^{\beta}_{\overline{\alpha}+(n+1)-i}(DU)d\widetilde{x}\\
&+(-1)^{n-1}\sum_{j\in \overline{\beta}}(-1)^{|\alpha|}\sigma(j,\beta) \sigma(\alpha,\overline{\alpha})\int_{\Omega\times (0,1)}\partial_{y_j}\Psi (\widetilde{x},U(\widetilde{x}))M^{\beta+j}_{\overline{\alpha}+(n+1)}(DU)d\widetilde{x}\\
&=-\sum_{i\in \overline{\alpha}+(n+1)} \sigma(\overline{\alpha}+(n+1)-i,i)\sigma(\alpha,\overline{\alpha}) \int_{\Omega\times (0,1)} \partial_{x_i}\Psi(\widetilde{x},U(\widetilde{x}))M^{\beta}_{\overline{\alpha}+(n+1)-i}(DU)d\widetilde{x}\\
&-\sum_{j\in \overline{\beta}}\sigma(\beta,j) \sigma(\alpha,\overline{\alpha})\int_{\Omega\times (0,1)} \partial_{y_j}\Psi (\widetilde{x},U(\widetilde{x}))M^{\beta+j}_{\overline{\alpha}+(n+1)}(DU)d\widetilde{x}\\
&=-\sum_{i\in \overline{\alpha}+(n+1)} \sigma(\overline{\alpha}+(n+1)-i,i)\sigma(\alpha,\overline{\alpha})\int_{\Omega\times (0,1)}
\left(\partial_{x_i}\Psi(\widetilde{x},U(\widetilde{x}))+\sum_{j\in \overline{\beta}}\partial_{y_j}\Psi (\widetilde{x},U(\widetilde{x})) \partial_{x_i}U^j\right)\\
&\cdot M^{\beta}_{\overline{\alpha}+(n+1)-i}(DU)d\widetilde{x}\\
&=-\sum_{i\in \overline{\alpha}+(n+1)} \sigma(\overline{\alpha}+(n+1)-i,i)\sigma(\alpha,\overline{\alpha})\int_{\Omega\times (0,1)}
D_{x_i} [\Psi(\widetilde{x},U(\widetilde{x}))] M^{\beta}_{\overline{\alpha}+(n+1)-i}(DU)d\widetilde{x}.
\end{align*}
Note that the last equality holds since
$$\sum_{i\in \overline{\alpha}+(n+1)} \sigma(\overline{\alpha}+(n+1)-i,i)\partial_{x_i}U^j M^{\beta}_{\overline{\alpha}+(n+1)-i}(DU)=M_{\overline{\alpha}+(n+1)}^{\beta+(n+1)}(DU(j))=0 ~~~~\mbox{for any}~ j\in \beta.$$
Where $U(j)=(U_1,U_2,\cdot\cdot\cdot,U_j,\cdot\cdot\cdot,U_N,U_j)$.
In particular
\begin{align*}
(T_u^{p-1})^{\overline{0}0}(\psi)
&=-\int_{\Omega\times (0,1)}\frac{\partial}{\partial x_{n+1}} \Psi(\widetilde{x},U(\widetilde{x}))d\widetilde{x}
=\int_{\Omega}\psi(x,u(x))dx.
\end{align*}
Therefore the proof is completed.
\end{proof}

Aa a consequence, it is easy to check that
\begin{corollary}[Weak  continuity theorem I]\label{subwct001}
Let $\{u_j\}_{j=1}^{\infty},u\subset W^{1-\frac{1}{p},p}(\Omega,\mathbb{R}^N)$ with $2\leq p\leq n'$ and $u_j\rightarrow u$ in $W^{1-\frac{1}{p},p}(\Omega,\mathbb{R}^N)$. Then
$$T^{p-1}_{u_j}\rightharpoonup T^{p-1}_u~~~~\mbox{in}~D_{n,p-1}(\Omega\times \mathbb{R}^N).$$
\end{corollary}

In the sequel we will let $\pi:\mathbb{R}^n\times \mathbb{R}^N\rightarrow \mathbb{R}^n$ and $\widehat{\pi}:\mathbb{R}^n\times \mathbb{R}^N\rightarrow \mathbb{R}^N$ denote the orthogonal projection onto the $x$ and $y$ coordinates, respectively.
In the following proposition, we  establish the relation between the  distribution minors and semi-currents in fractional Sobolev spaces.
\begin{proposition}\label{subcor42}
Let $u\in W^{1-\frac{1}{p},p}(\Omega,\mathbb{R}^N)$ with $2\leq p\leq n'$, then
$$\mbox{Div}_{\overline{\alpha}}^{\beta}(Du)=\sigma(\alpha,\overline{\alpha})\pi_{\sharp}(T^{p-1}_u)^{\alpha\beta}$$
for any  $\alpha\in I(n-k,n)$, $\beta\in I(k,N)$ with $0\leq k\leq p-1$.
\end{proposition}
\begin{proof}
Let $\psi\in C_c^{\infty}(\Omega)$ and $\Psi\in C_c^{\infty}(\Omega\times[0,1))$ be a extension of $\psi$. Fix $R>0$, we choose a cut-off function $\chi_R\in  C_c^{\infty}(\mathbb{R}^N)$ such that
\begin{equation}\label{subequ401}
0\leq \chi_R\leq 1, ~\chi_R=1~\mbox{in}~B(0,R),~~\chi_R=0~\mbox{in}~\mathbb{R}^n\backslash B(0,2R), ~~|D\chi_R|\leq \frac{2}{R}.
\end{equation}
Combining (\ref{subfor43}) with  (\ref{sub1002}) we obtain
\begin{align*}
&\sigma(\alpha,\overline{\alpha})\pi_{\sharp}(T^{p-1}_u)^{\alpha\beta}(\psi)=\sigma(\alpha,\overline{\alpha})\lim_{R\rightarrow +\infty}T^{p-1}_u(\psi(x) \chi_R(y)dx^{\alpha}\wedge dy^{\beta})\\
&=\lim_{R\rightarrow\infty}  -\sum_{i\in \overline{\alpha}+(n+1)} \sigma(\overline{\alpha}+(n+1)-i,i) \int_{\Omega\times (0,1)} \partial_{x_i} \Psi(\widetilde{x})\chi_R(U(\widetilde{x}))M^{\beta}_{\overline{\alpha}+(n+1)-i}(DU(\widetilde{x}))d\widetilde{x}\\
&-\lim_{R\rightarrow\infty}\sum_{j\in \overline{\beta}}\sigma(\beta,j) \int_{\Omega\times (0,1)} \Psi(\widetilde{x}) \partial_{y_j}\chi_R(U(\widetilde{x}))M^{\beta+j}_{\overline{\alpha}+(n+1)}(DU(\widetilde{x}))d\widetilde{x}\\
&=-\sum_{i\in \overline{\alpha}+(n+1)} \sigma(\overline{\alpha}+(n+1)-i,i) \int_{\Omega\times (0,1)} M^{\beta}_{\overline{\alpha}+(n+1)-i}(DU)\partial_{i}\Psi d\widetilde{x},
\end{align*}
which completes the proof.
\end{proof}

In particular if $u\in W^{1-\frac{1}{n'},n'}(\Omega,\mathbb{R}^N)$, the action of the semi-current $T^{n'-1}_u$ from Definition \ref{subdef41} can be extended to a current $T_u\in \mathcal{D}_n(\Omega\times \mathbb{R}^N)$ carried by the graph of $u$.
\begin{definition}\label{subdef42}
Let $u\in W^{1-\frac{1}{n'},n'}(\Omega,\mathbb{R}^N)$. Then the $n$-dimensional current $T_u\in\mathcal{D}_n(\Omega\times \mathbb{R}^N)$ is defined by
$$T_u:=(-1)^{n-1}(\partial G_U)\llcorner \Omega\times \mathbb{R}^N,$$
where $U\in W^{1,n'}(\Omega\times (0,1), \mathbb{R}^N)$ is any extension of $u$.
I.e., For any $\omega\in \mathcal{D}^n(\Omega\times \mathbb{R}^N)$,
$$T_u(\omega)=(-1)^{n-1}G_U(d\widetilde{\omega}),$$
for any extension $\widetilde{\omega}\in \mathcal{D}^n(\Omega\times [0,1)\times \mathbb{R}^N)$ of $\omega$.
\end{definition}
Similar to Definition \ref{subdef41},  the  quantity $T_u$ is well-defined since the trace theory and the following lemma in \cite[Vol. I, Sec. 3.2.5]{GMS}
\begin{lemma}\label{sublem41}
Let $\Omega$ be a bounded Lipschitz domain in $\mathbb{R}^n$ and $U,V\in W^{1,\underline{n}}(\Omega,\mathbb{R}^N)$. Then $U=V$ on $\partial \Omega$ implies
$$\partial G_U=\partial G_V.$$
\end{lemma}

The arguments similar to the one used in Proposition \ref{subpo32} and Proposition \ref{subcor42} shows that
\begin{proposition}\label{subthe412}
Let $u\in W^{1-\frac{1}{n'},n'}(\Omega,\mathbb{R}^N)$, $T^{n'-1}_u$ and $T_u$ be given as above. Then
\begin{enumerate}
\item[{\em(\romannumeral1)}]  $\partial T_u \llcorner \Omega\times \mathbb{R}^N=0$.
\item [{\em(\romannumeral2)}] $T_u=T^{n'-1}_u+S_u$ such that $S_u\in \mathcal{D}^n(\Omega\times \mathbb{R}^N)$ is defined by
\begin{equation}\label{subfor44}
\begin{split}
S^{\alpha\beta}_u(\psi )=\begin{cases}
0,~~~~~~~~~~~|\beta|\leq n'-1, \\
A_u^{\alpha\beta}(\psi),~~~~ n'-1<|\beta|\leq \underline{n}.\\
\end{cases},
\end{split}
\end{equation}
 where $\alpha\in I(n-k,n)$, $\beta\in I(k,N)$ with $0\leq k\leq \underline{n}$, $\psi\in C^{\infty}_c(\Omega\times \mathbb{R}^N)$ and
\begin{equation}\label{subfor445}
A_u^{\alpha\beta}(\psi):=-\sigma(\alpha,\overline{\alpha})\sum_{i\in \overline{\alpha}+(n+1)} \sigma(\overline{\alpha}+(n+1)-i,i)\int_{\Omega\times (0,1)}
\partial_{x_i}\Psi(\widetilde{x},U(\widetilde{x})) M^{\beta}_{\overline{\alpha}+(n+1)-i}(DU(\widetilde{x}))d\widetilde{x}
\end{equation}
for any extension $U\in W^{1,n'}(\Omega\times (0,1), \mathbb{R}^N)$ and $\Psi\in C^{\infty}_c(\Omega\times [0,1)\times \mathbb{R}^N)$ of $u$ and
$\psi$, respectively.
\item [{\em(\romannumeral3)}]$\mbox{Div}_{\overline{\alpha}}^{\beta}(Du)=\sigma(\alpha,\overline{\alpha})\pi_{\sharp}T_u^{\alpha\beta}$
for any  $\alpha\in I(n-k,n)$, $\beta\in I(k,N)$ with $0\leq k\leq \underline{n}$.
\end{enumerate}
\end{proposition}

\begin{remark}
  In fact the current $S_u=0$ in case $n<N$ and  the equality (\ref{subfor43}) is still true for any   $\alpha\in I(n-k,n)$ and $\beta\in I(k,N)$ with $k= N$ in case $n\geq N$. Therefore the current $T_u$ can be written as
\begin{equation}\label{sub4001}
\begin{split}
T_u(\omega)=\sum_{|\alpha|+|\beta|=n} \sigma(\alpha,\overline{\alpha}) \langle\mbox{Div}_{\overline{\alpha}}^{\beta}(Du(x)), \omega_{\alpha\beta}(x,u(x))\rangle,
\end{split}
\end{equation}
where $\omega(x,y)=\sum_{|\alpha|+|\beta|=n} \omega_{\alpha\beta}(x,y) dx^{\alpha}\wedge dy^{\beta}$ and
\begin{equation}
\langle\mbox{Div}_{\overline{\alpha}}^{\beta}(Du(x)), \omega_{\alpha\beta}(x,u(x))\rangle:= -\sum_{i\in \overline{\alpha}+(n+1)} \sigma(\overline{\alpha}+(n+1)-i,i)\int_{\Omega\times (0,1)}
D_{x_i} [\widetilde{\omega}_{\alpha\beta}(\widetilde{x},U(\widetilde{x}))] M^{\beta}_{\overline{\alpha}+(n+1)-i}(DU)d\widetilde{x}.
\end{equation}
 Which means that the current $T_u$ is determined by all distribution Jacobian  minors of $u$.
\end{remark}
The following weak continuity result may be proved in the same way as Corollary \ref{subwct001}.

\begin{proposition}\label{subpro403}
Let $\{u_j\}_{j=1}^{\infty},u\subset W^{1-\frac{1}{n'},n'}(\Omega,\mathbb{R}^N)$ with $u_j\rightarrow u$ in $W^{1-\frac{1}{n'},n'}(\Omega,\mathbb{R}^N)$. Then
$$T_{u_j}\rightharpoonup T_u~~~~\mbox{in}~D_n(\Omega\times \mathbb{R}^N).$$
\end{proposition}

Next we will show that the current $T_u$   is consistent with the current $G_u$ defined in the class $\mathcal{A}^1(\Omega,\mathbb{R}^n)$ when $u\in W^{1,n'}(\Omega, \mathbb{R}^N)$. In order to prove the result, we need the knowledge of   boundary current $G_{u,\partial \Omega}$, which introduced by Giaquinta-Modica-Sou\v{c}ek in \cite[Vol. I, Sec. 3.2.5]{GMS}. Recall that for any  $u\in \mathcal{A}^1(\partial \Omega, \mathbb{R}^N)$, the i.m. rectifiable $(n-1)$-current $G_{u,\partial \Omega}$ is defined for $\omega\in \mathcal{D}^{n-1}(\partial\Omega \times \mathbb{R}^N) $ by
\begin{equation}\label{subfor45}
G_{u,\partial \Omega}(\omega):=\int_{\partial \Omega} \langle \omega (x,u(x)), A(Du)\rangle d\mathcal{H}^{n-1},
\end{equation}
where $A(Du(x))$ denotes the $(n-1)$-vector in $\mathbb{R}^{n+N}$ given by
\begin{equation}\label{subfor44}
\begin{split}
A(Du(x))&=\sum_{|\alpha|+|\beta|=n-1} A^{\alpha\beta}(Du(x)) e_{\alpha}\wedge\epsilon_{\beta}\\
A^{\alpha\beta}(Du(x)):&=(-1)^{|\alpha|}\sigma(\alpha,\overline{\alpha}) \sum_{i\in\overline{\alpha}}\sigma(i,\overline{\alpha}-i) M^{\beta}_{\overline{\alpha}-i}(Du(x)) \nu_i(x).
\end{split}
\end{equation}
and $\nu(x)=\sum_{i=1}^{n}\nu_i(x)e_i$ is the exterior unit normal vector at $x\in \partial \Omega$ for $\mathcal{H}^{n-1}$-a.e.

\begin{proposition}\label{sub412}
Let $u\in W^{1,n'}(\Omega,\mathbb{R}^N)$. Then the current $T_u$ defined in Definition \ref{subdef42} is  consistent with the current $G_u$ defined in the class $\mathcal{A}^1(\Omega,\mathbb{R}^N)$.
\end{proposition}
\begin{proof}
According to Proposition \ref{subpro403}, it is  suffices to prove the theorem in the case that $u$ is $C^1$ in $\Omega$.
For any $\alpha\in I(n-k,n)$, $\beta\in I(k,N)$ with $0\leq k\leq n'$ and  $\psi\in C^{\infty}_c(\Omega\times \mathbb{R}^N)$, we choose
a extension $U\in C^2(\Omega\times (0,1), \mathbb{R}^N)\cap C^1(\Omega\times [0,1), \mathbb{R}^N)$ and $\Psi\in C^{\infty}_c(\Omega\times [0,1)\times \mathbb{R}^N)$ of $u$ and
$\psi$, respectively. According to the Theorem 1 in \cite[Vol. I, Sec. 3.2.5]{GMS}, it follows that
$$\partial G_U=G_{U,\partial\widetilde{\Omega}},$$
where $\widetilde{\Omega}=\Omega\times (0,1)$. Then the formula (\ref{subfor45}) and (\ref{subfor44}) implies that
\begin{align*}
&T_u(\psi dx^{\alpha}\wedge dy^{\beta})=(-1)^{n-1}(\partial G_U)(\Psi dx^{\alpha}\wedge dy^{\beta})\\
&=(-1)^{n-1}\int_{\partial \widetilde{\Omega}} \langle \Psi (x,U(x)), A^{\alpha \beta}(DU)\rangle d\mathcal{H}^{n}\\
&=(-1)^{k-1}\sigma(\alpha,\overline{\alpha}+(n+1))\int_{\Omega \times\{0\}} \sum_{i\in \overline{\alpha}+(n+1)}\sigma(i,\overline{\alpha}+(n+1)-i)M_{\overline{\alpha}+(n+1)-i}^{\beta}(DU(\widetilde{x})) \nu_i(\widetilde{x})\Psi(\widetilde{x},U(\widetilde{x})) d\mathcal{H}^{n}\\
&=\sigma(\alpha,\overline{\alpha})\int_{\Omega} M_{\overline{\alpha}}^{\beta}(Du(x))\psi(x,u(x)) dx.
\end{align*}
Note that the last equality holds since $\nu=-e_{n+1}$ for $\mathcal{H}^n$-a.e. $\widetilde{x}\in \Omega\times \{0\}$.
\end{proof}

\begin{proof}[\bf Proof of Theorem \ref{subdetthm002}]
Having established Proposition \ref{subthe412}, \ref{subpro403} and \ref{sub412}, we can now conclude the proof of Theorem \ref{subdetthm002}.
\end{proof}

Let $T$ be a current in $D_n(\Omega\times \mathbb{R}^N)$ whose $\overline{0}0$-component $T^{\overline{0}0}$ is a Radon measure with bounded variation in $\Omega\times \mathbb{R}^N$, i.e.
$$\|T\|_0:=\sup\{T(\psi(x,y) dx)\mid \psi\in C_c^0(\Omega\times \mathbb{R}^N)),|\psi|\leq 1\}<\infty.$$
 Then we define the $L^1$-norm of $T$ by
$$\|T\|_1:=\sup\{T(\psi(x,y)|y|dx)\mid \psi\in C_c^0(\Omega\times \mathbb{R}^N)),|\psi|\leq 1\}.$$
Hence we have
\begin{proposition}
Let $u\in W^{1-\frac{1}{n'},n'}(\Omega,\mathbb{R}^N)$ and $T_u$ be given as above. Then
$$\|T_u\|_1<\infty,~~~~\pi_{\sharp}T_u=[\![\Omega]\!].$$
\end{proposition}
\begin{proof}
Clearly there exists a sequence $u_j\in W^{1-\frac{1}{n'},n'}(\Omega,\mathbb{R}^N)\cap C^{\infty}(\overline{\Omega},\mathbb{R}^N)$ such that $u_j\rightarrow u$ in $W^{1-\frac{1}{n'},n'}(\Omega,\mathbb{R}^N)$.
Due to  Proposition \ref{subpro403} and \ref{sub412},  it follows that
$$T_{u_j}=G_{u_j}\rightharpoonup T_u~~~~\mbox{in}~D_n(\Omega\times \mathbb{R}^N).$$
For $T_{u_j}$, since $u_j$ are smooth, we trivially have
$$\|T_{u_j}\|_1=\int_{\Omega} |u_j|dx,~~~~\pi_{\sharp}T_{u_j}=[\![\Omega]\!].$$
Hence
$$\|T_{u}\|_1\leq \liminf_{j\rightarrow \infty}\|T_{u_j}\|_1<\infty.$$
Fix $R>0$, a cut-off function  $\chi_R(y)$ as in (\ref{subequ401}) and $\psi\in C_c^{\infty}(\Omega)$
\begin{align*}
&|\pi_{\sharp}T_u(\psi(x)dx)-\pi_{\sharp}T_{u_j}(\psi(x)dx)|\\
&\leq |T_u(\psi(x)\chi_{R}(y)dx)-T_{u_j}(\psi(x)\chi_{R}(y)dx)|+|T_u(\psi(x)(1-\chi_{R}(y))dx)-T_{u_j}(\psi(x)(1-\chi_{R}(y))dx)|\\
&\leq |T_u(\psi(x)\chi_{R}(y)dx)-T_{u_j}(\psi(x)\chi_{R}(y)dx)|+(\|T^{\overline{0}0}_{u}\|+\|T^{\overline{0}0}_{u_j}\|)(|\psi|\frac{|y|}{R})\\
&\leq  |T_u(\psi(x)\chi_{R}(y)dx)-T_{u_j}(\psi(x)\chi_{R}(y)dx)|+\frac{\|\psi\|_{\infty}(\|T_{u}\|_1+\|T_{u_j}\|_1)}{R},
\end{align*}
 which implies $\pi_{\sharp}T_u=[\![\Omega]\!]$ and hence the proof is completed.
\end{proof}

It is well known that in general, the distribution determinant $\mbox{Det} Du$ is not a function and
$$\mbox{Det} Du\neq \det Du.$$
For instance, if $u(x):=\frac{x}{|x|}$ one easily verifies that $\det Du=0$ and $\mbox{Det} Du= |B(0,1)|\delta_0$. But M\"{u}ller and  Spector showed that the absolutely continuous part of $\mbox{Det} Du$ with respect to Lebesgue measure is equal to the pointwise determinant $\det Du$ if $u$ satisfies the certain condition in some Sobolev spaces (see \cite[Vol. I, Sec. 3.2.4]{GMS}).
In Theorem \ref{subdetthm003}, we establish a similar result for distribution minor in  fractional Sobolev spaces.

\begin{proof}[\bf Proof of Theorem \ref{subdetthm003}]
 It follows from the boundary rectifiability  theorem in \cite[Sec. 3]{SIM} that $T_u$ is also an i.m. rectifiable current.  Combing with  the above results,  we can easily obtain that $T_u\in \cart(\Omega\times \mathbb{R}^N)$.
 According to the structure theorem in \cite[Vol. I, Sec. 4.2.3]{GMS}, there exists a function $u_T\in BV(\Omega, \mathbb{R}^N)\cap \mathcal{A}^1(\Omega, \mathbb{R}^N)$ such that
 $$T_u(\psi(x,y)dx)=\int_{\Omega} \psi(x,u_T(x))dx,$$
 where $\psi$ is a continuous functions satisfying $|\psi(x,y)|\leq c(1+|y|)$. Hence $u=u_T$ a.e. $x\in \Omega$ which implies $u\in BV(\Omega, \mathbb{R}^N)\cap \mathcal{A}^1(\Omega, \mathbb{R}^N)$

 From Theorem 4 in  \cite[Vol. I, Sec. 4.2.3]{GMS} we know that
 $$M_{\alpha}^{\beta}(T_u)^{\mbox{ac}}(x)=M_{\alpha}^{\beta}( ap Du_T(x)), ~~\mbox{for}~\mathcal{L}^n~ \mbox{a.e.}~x\in \Omega$$
 Combing with  (\romannumeral3) in Proposition \ref{subthe412} we obtain (\ref{subdetformula0031}), which completes the proof of Theorem.
\end{proof}

 As a consequence, we also have
 \begin{corollary}
 Let  $u\in W^{1-\frac{1}{p},p}(\Omega,\mathbb{R}^N)$ with $2\leq p\leq N$ and $\gamma\in I(p,N)$. If  $\mathbb{M}(T_{u^{\gamma}})<\infty$ where $u^{\gamma}:=(u^{\gamma_1},\cdot\cdot\cdot, u^{\gamma_{p}})$, then $u^{\gamma}\in BV(\Omega,\mathbb{R}^{\gamma})\cap \mathcal{A}^1(\Omega,\mathbb{R}^{\gamma})$ and
 $$\mbox{Div}_{\overline{\alpha}}^{\beta}(Du)^{\mbox{ac}}(x)=M_{\overline{\alpha}}^{\beta}( ap Du(x)), ~~\mbox{for a.e.}~x\in \Omega,$$
 for  any  $\alpha\in I(n-k,n)$, $\beta\in I(k,N)$ with $\beta\subset \gamma$ and $0\leq k\leq p$.
 \end{corollary}

\section{BNV functions in fractional Sobolev spaces}

In this section we extend the notion of  bounded N-variation functions, which was established by Jerrard-Soner \cite{JS} in $W^{1,N-1}(\Omega, \mathbb{R}^N)\cap L^{\infty}(\Omega, \mathbb{R}^N)$,  to the fractional Sobolev space $W^{1-\frac{1}{N},N}(\Omega, \mathbb{R}^N)$ with $n\geq N\geq 2$.
We shall then discuss some properties ,such as coarea formula, chain rule and so on.
%In [], R. Jerrard and H. Soner defined the distributional Jacobian $[Ju]$ of $u: \Omega\subset \mathbb{R}^m\rightarrow \mathbb{R}^n$ ($m\geq n$)

\begin{definition}\label{subdef501}
Let $u\in W^{1-\frac{1}{N},N}(\Omega, \mathbb{R}^N)$ with $2\leq N\leq n$. We can associate the distributional Jacobian $[Ju]$ as the $n-N$-dimensional current  in $D_{n-N}(\Omega)$
defined for $\omega\in \mathcal{D}^{n-N}(\Omega)$ by
\begin{equation}\label{subdef501}
[Ju](\omega):=\sum_{\alpha\in I(n-N,n)}\sigma(\alpha,\overline{\alpha}) \langle\mbox{Div}^{\overline{0}}_{\overline{\alpha}}(Du), \omega_{\alpha}\rangle,
\end{equation}
where $\mbox{Div}^{\overline{0}}_{\overline{\alpha}}(Du)$ are given by Definition \ref{subdetdef001}  and $\omega=\sum_{\alpha\in I(n-N,n)}\omega_{\alpha}(x)dx^{\alpha}$.
\end{definition}

According to the weak  continuity result of  distribution minors, it follows that
\begin{proposition}\label{subpro501}
Let $\{u_j\}_{j=1}^{\infty},u\subset W^{1-\frac{1}{N},N}(\Omega,\mathbb{R}^N)$  with $2\leq N\leq n$.  If $u_j\rightarrow u$ in $W^{1-\frac{1}{N},N}(\Omega,\mathbb{R}^N)$, then
$$[Ju_j]\rightharpoonup [Ju]~~\mbox{in}~\mathcal{D}_{n-N}(\Omega).$$
\end{proposition}

In the following lemma, we show that  $[Ju]$ can be represented by the $n$-current $T_u$ which is introduced in Section 4.

\begin{proposition}\label{subpro511}
Let $u\in W^{1-\frac{1}{N},N}(\Omega, \mathbb{R}^N)$  with $2\leq N\leq n$. Then
$$[Ju]=(-1)^{(n-N)N}\pi_{\sharp}(T_u\llcorner \widehat{\pi}^{\sharp}dy)=(-1)^{(n-N)(N-1)+1}\partial\left(\pi_{\sharp}(G_U\llcorner \widehat{\pi}^{\sharp}dy)\right)\llcorner \Omega.$$
for any extension $U\in W^{1,N}(\Omega\times (0,1))$. More precisely, for any $\omega\in \mathcal{D}^{n-N}(\Omega)$
$$[Ju](\omega)=(-1)^{(n-N)(N-1)+1}\pi_{\sharp}(G_U\llcorner \widehat{\pi}^{\sharp}dy) (d\widetilde{\omega}),$$
where $\widetilde{\omega}\in \mathcal{D}^{n-N}(\Omega\times [0,1))$ is a extension of $\omega$.
\end{proposition}

\begin{proof}
Our proof starts with the observation that
 the boundary $\partial G_U$ has support in $\partial(\Omega \times (0,1))\times \mathbb{R}^N$ since $\partial G_U \llcorner \Omega \times (0,1)\times \mathbb{R}^N=0$.
 Then
 $$\pi_{\sharp}\left((\partial G_U)\llcorner \widehat{\pi}^{\sharp}dy\right)\in \mathcal{D}^{n-N}(\mathbb{R}^{n+1}), ~~\mbox{spt}~\left(\pi_{\sharp}\left((\partial G_U)\llcorner \widehat{\pi}^{\sharp}dy\right)\right)\subset \partial(\Omega \times (0,1)),$$
  which implies that $\pi_{\sharp}\left((\partial G_U)\llcorner \widehat{\pi}^{\sharp}dy\right)\llcorner \Omega\times\{0\}$ can be seen as a $n-N$-current in $\Omega$.
Moreover it is easy to check that $(-1)^{(n-N)N}\pi_{\sharp}(T_u\llcorner \widehat{\pi}^{\sharp}dy)=(-1)^{(n-N)(N-1)+1}\partial\left(\pi_{\sharp}(G_U\llcorner \widehat{\pi}^{\sharp}dy)\right)\llcorner \Omega$.

Fix $\alpha \in I(n-N,n)$, $\psi(x)\in C^{\infty}_c(\Omega)$ and a extension $\Psi\in C^{\infty}_c(\Omega\times [0,1))$ of $\psi$, we choose a cut-off function $\chi_R\in  C_c^{\infty}(\mathbb{R}^N)$ such as (\ref{subequ401}), then
\begin{align*}
&(-1)^{(n-N)(N-1)+1}\pi_{\sharp}(G_U\llcorner \widehat{\pi}^{\sharp}dy) (d(\Psi(\widetilde{x}) dx^{\alpha}))\\
&=\lim_{R\rightarrow \infty}(-1)^{(n-N)(N-1)+1}(G_U\llcorner \widehat{\pi}^{\sharp}dy)(d\Psi(\widetilde{x})\chi_R(y)\wedge dx^{\alpha})\\
&=\lim_{R\rightarrow \infty}(-1)^{n-1}G_{U}(d\Psi(\widetilde{x}) \chi_R(y)\wedge dx^{\alpha}\wedge dy)\\
%&=\sum_{i\in \overline{\alpha}+(n+1)}(-1)^{n-1}\sigma(i,\alpha)\int_{\Omega\times (0,1)}\partial_i\Psi dx^{\alpha+i}\wedge dU\\
&=\lim_{R\rightarrow \infty}\sum_{i\in \overline{\alpha}+(n+1)}(-1)^{n-1}\sigma(i,\alpha)\sigma(\alpha+i,\overline{\alpha}-i+(n+1))\int_{\Omega\times (0,1)} M^{\overline{0}}_{\overline{\alpha}-i+(n+1)}(DU(\widetilde{x}))\chi_R(U(\widetilde{x}))\partial_i\Psi(\widetilde{x}) d\widetilde{x}\\
&= -\sigma(\alpha,\overline{\alpha})\sum_{i\in \overline{\alpha}+(n+1)}\sigma(\overline{\alpha}-i+(n+1),i)\int_{\Omega\times (0,1)} M^{\overline{0}}_{\overline{\alpha}-i+(n+1)}(DU(\widetilde{x})) \partial_i\Psi(\widetilde{x}) d\widetilde{x}\\
&=\sigma(\alpha,\overline{\alpha})\langle\mbox{Div}^{\overline{0}}_{\overline{\alpha}}(Du), \psi\rangle,
\end{align*}
which completes the proof.
\end{proof}

Base on Definition \ref{subdef501}, we can give a definition of  functions of  bounded higher variation in the   fractional Sobolev space.
\begin{definition}\label{subdef502}
A function $u\in W^{1-\frac{1}{N},N}(\Omega, \mathbb{R}^N)$  with $2\leq N\leq n$ has bounded $N$-variation in $\Omega\subset \mathbb{R}^{n}$ if
$$\mathbf{M}([Ju]):=\sup \{[Ju](\omega)\mid \omega \in D^{n-N}(\Omega),\|\omega\|\leq 1\}<\infty.$$
We write $BNV(\Omega)$ to denote the space of functions of bounded $N$-variation in $\Omega$.
\end{definition}
\begin{remark}
We note that Jerrard-Soner's definition of $BNV$ in $L^{\infty}\cap W^{1,N-1}(\Omega, \mathbb{R}^N)$ is contained in ours in view of the embedding theorem.
\end{remark}

An immediate consequence of Riesz theorem for measures is that
\begin{proposition}
Let $u\in BNV(\Omega, \mathbb{R}^N)$. Then there exist a Radon measure $\left\|[Ju]\right\|$ on $\Omega$ and a $\left\|[Ju]\right\|$-measurable function $\nu: \Omega\rightarrow \bigwedge_{n-N} \mathbb{R}^n$ with $\|\nu\|=1$ for $\left\|[Ju]\right\|$-a.e. such that
$$[Ju](\omega)=\int_{\Omega}\langle\omega,\nu\rangle d\left\|[Ju]\right\|$$
for any $\omega\in \mathcal{D}^{n-N}(\Omega)$.
\end{proposition}

Due to the results of Proposition \ref{subpro501}, it follows that
\begin{proposition}
Suppose  $\{u_j\}_{j=1}^{\infty} \subset BNV(\Omega,\mathbb{R}^N)$ and $u_j\rightarrow u$ in $W^{1-\frac{1}{N},N}(\Omega,\mathbb{R}^N)$. Then
$$\left\|[Ju]\right\|(\Omega)\leq \liminf_{j\rightarrow \infty}\left\|[Ju_j]\right\|(\Omega).$$
\end{proposition}

Many results about $BV$ have some sort of generalization in $BNV$, such as Jerrard-Soner \cite{JS} established general versions of the chain rule and the coarea formula for $BNV$, which is defined in $L^{\infty}\cap W^{1,N}(\Omega, \mathbb{R}^N)$.
 Our goal is to generalize the results from $BV$ to the space $BNV$ which is defined in $ W^{1-\frac{1}{N},N}$.
 Before this, we recall that if $U\in W^{1,1}(\Omega)$, then $U$ is approximately differentiable at $x\in R_U$ with   approximate differential given by the Lebesgue value $DU(x)$ of the distributional gradient $DU$ at x, where $R_U:=\{x\in \Omega \mid  x\in \mathcal{L}_U\cap A_D(U)\}$,  $\mathcal{L}_u$ is the set of Lebesgue points, $A_D(u)$ is the set of approximate differentiability points of $u$. Moreover we denote by $Ju$ the Jacobian of $Du$.
  Our  first   state is the following lemma

 \begin{lemma}\label{sublem501}
 Let $u\in W^{1,N}(\Omega, \mathbb{R}^N)$  with $2\leq N\leq n$. Then for $\mathcal{H}^N$-a.e. $y\in \mathbb{R}^N$,
there exist an integer multiplicity $G_yu=\tau(u^{-1}(y)\cap E_u,1,\zeta)\in \mathcal{R}_{n-N}(\Omega)$  such that
$$\pi_{\sharp}\left( G_u\llcorner \widehat{\pi}^{\sharp} dy\right)(\omega)=\int_{\mathbb{R}^N} \left\{\int_{u^{-1}(y)\cap E_u}\langle\omega ,\zeta\rangle d\mathcal{H}^{n-N} \right\}dy$$
for any $\omega\in \mathcal{D}^{n-N}(\Omega)$. Where $E_u:=R_u\cap \{x\in \Omega \mid Ju(x)>0\}$.
 \end{lemma}
 \begin{proof}
 Notice that the i.m. rectifiable n-current $G_u$ can be written as $G_u=\tau(\mathcal{M},1,\xi_u)$, where
$$\mathcal{M}:=\{(x,u(x))\mid x\in R_u \cap \Omega\}, ~~~~ \xi_u(x,u(x))=\frac{\xi_1\wedge\cdot\cdot\cdot\wedge\xi_{n}}{|\xi_1\wedge\cdot\cdot\cdot\wedge\xi_{n}|}$$
with $\xi_i=e_i+\sum_{s=1}^ND_iu^s \varepsilon_s$.
For any $\mathcal{H}^{n}$-a.e. $(x,y)\in \mathcal{M}$, the linear map $d^{\mathcal{M}}\widehat{\pi}_{(x,y)}: T_{(x,y)}\mathcal{M} \rightarrow \mathbb{R}^N$ can be written as
\begin{equation}\label{subbnv005}
d^{\mathcal{M}}\widehat{\pi}_{(x,y)}(\xi_i):=\sum_{s=i}^N(\xi_i\cdot \nabla^{\mathcal{M}}\widehat{\pi}^s) \varepsilon_s=\sum_{s=1}^ND_iu^s(x)\varepsilon_s
\end{equation}
for any $i=1,\cdot\cdot\cdot, n$.
%\begin{align*}
%J^{\mathcal{M}}_{\widehat{\pi}}(x,y)&=\left|\bigwedge_{n} d^{\mathcal{M}}\widehat{\pi}(\xi_u(x,y))\right|=\frac{1}{|\xi_1\wedge\cdot\cdot\cdot\wedge\xi_{n}|}\left|d^{\mathcal{M}}\widehat{\pi}(\xi_1)\wedge\cdot\cdot\cdot d^{\mathcal{M}}\widehat{\pi}(\xi_{n})\right|\\
%&=\frac{1}{|\xi_1\wedge\cdot\cdot\cdot\wedge\xi_{n}|}\left|\sum_{s=1}^N(\xi_1\cdot \nabla^{\mathcal{M}}\widehat{\pi}^s)\varepsilon_s\wedge\cdot\cdot\cdot \sum_{s=1}^N(\xi_{n}\cdot \nabla^{\mathcal{M}}\widehat{\pi}^s)\varepsilon_s\right|\\
%&=\frac{1}{|\xi_1\wedge\cdot\cdot\cdot\wedge\xi_{n}|}\left|\sum_{s=1}^ND_1u^s\varepsilon_s\wedge\cdot\cdot\cdot \sum_{s=1}^ND_{n}u^s\varepsilon_s\right|\\
%&=\frac{1}{|\xi_1\wedge\cdot\cdot\cdot\wedge\xi_{n}|}\left|\sum_{\alpha\in I(n-N,n)}M^{\overline{0}}_{\overline{\alpha}}(Du)\varepsilon_{\alpha}\right|
%\end{align*}
Let
$$\mathcal{M}_{+}:=\{(x,y))\in \mathcal{M}\mid J^{\mathcal{M}}_{\widehat{\pi}}(x,y):=\sqrt{\det\left((d^{\mathcal{M}}\widehat{\pi}_{(x,y)})(d^{\mathcal{M}}\widehat{\pi}_{(x,y)})^{\ast}\right)}>0\}.$$
(\ref{subbnv005}) and the Binet-Cauchy formula of determinant implies that
\begin{equation*}
\mathcal{M}_{+}=\{(x,y)\in \mathcal{M}\mid \sqrt{det(DuDu^{\ast})}>0\}=\{(x,y)\in \mathcal{M}\mid Ju(x)>0\}
\end{equation*}
According to Proposition 2 in \cite[Vol. I, Sect. 2.1.5]{GMS} and the fact $\pi(\widehat{\pi}^{-1}(y)\cap\mathcal{M}_{+})=u^{-1}(y)\cap E_{u}$,
$u^{-1}(y)\cap E_{u}$ is $\mathcal{H}^{n-N}$-measurable and countably $(n-N)$-rectifiable in $\Omega$  for $\mathcal{H}^{N}$-a.e. $y\in \mathbb{R}^N$.

 For $\mathcal{H}^{n}$-a.e. $(x,y)\in \mathcal{M}_{+}$, we factor
\begin{equation}\label{weakcoarea01}
\xi_u=\eta(x,y)\wedge\vartheta(x,y)
\end{equation}
so that $\eta(x,y)$ is a unit simple $N$-vector of $T_{(x,y)}\mathcal{M}$ and
$$\langle\eta(x,y),\widehat{\pi}^{\sharp}dy\rangle=J_{\widehat{\pi}}^{\mathcal{M}}(x,y),$$
$\vartheta(x,y)$ is a unit simple $(n-N)$-vector of  $T_{(x,y)}\mathcal{M}$ and
\begin{equation}\label{weakcoarea02}
\mbox{ker} ~d^{\mathcal{M}} \widehat{\pi}_{(x,y)} ~\mbox{is associated with}~\vartheta(x,y).
\end{equation}
Using the coarea formula one finds that if $\varpi\in \mathcal{D}^{n-N,0}(\Omega\times \mathbb{R}^{N})$, then
\begin{align*}
&G_u\llcorner (\widehat{\pi}^{\sharp}dy) (\varpi) =\int_{\mathcal{M}}\langle\widehat{\pi}^{\sharp}dy \wedge \varpi(x,y), \xi_u(x,y),\rangle d\mathcal{H}^{n}\\
&=\int_{\mathcal{M}} J_{\widehat{\pi}}^{\mathcal{M}}(x,y)\langle\varpi(x,y),\vartheta(x,y)\rangle d\mathcal{H}^{n}\\
&=\int_{\mathbb{R}^N}\left\{\int_{\widehat{\pi}^{-1}(z)\cap\mathcal{M}_{+}} \langle\varpi(x,y),\vartheta(x,y) \rangle d\mathcal{H}^{n-N}\right\}d\mathcal{H}^{N}(z)\\
&=\int_{\mathbb{R}^N}\tau(\widehat{\pi}^{-1}(z)\cap\mathcal{M}_{+},1,\vartheta)(\varpi) d\mathcal{H}^{N}(z).
\end{align*}
Obviously  $\tau(\widehat{\pi}^{-1}(z)\cap\mathcal{M}_{+},1,\vartheta)\in \mathcal{R}_{n-N}(\Omega\times \mathbb{R}^{N})$ for $\mathcal{H}^{N}$-a.e. $z\in \mathbb{R}^N$.
Set
$$\mathcal{M}_z=\widehat{\pi}^{-1}(z)\cap\mathcal{M}_{+}~~~~\mathcal{M}_z^{+}=\{(x,y)\in \mathcal{M}_z\mid J^{\mathcal{M}_z}_{\pi}(x,y)>0 \}.$$
Let $z\in \mathbb{R}^N$ such that  $\tau(\mathcal{M}_z,1,\vartheta)\in \mathcal{R}_{n-N}(\Omega\times \mathbb{R}^{N})$,
We claim that $\mathcal{M}_z^{+}=\mathcal{M}_z$ for $\mathcal{H}^{n-N}$ a.e. $(x,y)\in \mathcal{M}_z$.

Fix $(x,y)\in \mathcal{M}_z$ such that $T_{(x,y)}\mathcal{M}$ exists, (\ref{weakcoarea01}) implies
$$(\bigwedge_{n} d^{\mathcal{M}}\pi_{(x,y)})\xi_u(x,y)=(\bigwedge_{n} d^{\mathcal{M}}\pi_{(x,y)})(\eta(x,y)\wedge\vartheta(x,y)).$$
For simplicity ,we may set $d^{\mathcal{M}}\pi:=d^{\mathcal{M}}\pi_{(x,y)}$,  $\xi_u:=\xi_u(x,y),\eta:=\eta(x,y)$ and $\vartheta:=\vartheta(x,y)$. On the one hand
\begin{align*}
(\bigwedge_{n} d^{\mathcal{M}}\pi)\xi_u&=\frac{1}{|\xi_1\wedge\cdot\cdot\cdot\wedge\xi_{n}|}d^{\mathcal{M}}\pi(\xi_1)\wedge\cdot\cdot\cdot \wedge d^{\mathcal{M}}\widehat{\pi}(\xi_{n})\\
&=\frac{1}{|\xi_1\wedge\cdot\cdot\cdot\wedge\xi_{n}|}\sum_{s=1}^{n}(\xi_1\cdot \nabla^{\mathcal{M}}\pi^s)e_s\wedge\cdot\cdot\cdot\wedge \sum_{s=1}^{n}(\xi_{n}\cdot \nabla^{\mathcal{M}}\pi^s)e_s\\
&=\frac{1}{|\xi_1\wedge\cdot\cdot\cdot\wedge\xi_{n}|}\sum_{s=1}^{n}(\xi_1\cdot e_s)e_s\wedge\cdot\cdot\cdot\wedge \sum_{s=1}^{n}(\xi_{n}\cdot e_s)e_s\\
&=e_1\wedge\cdot\cdot\cdot\wedge e_{n}.
\end{align*}
On the other hand, $\eta$ and $\vartheta$ are simple $N$-vector and $(n-N)$-vector  in $R^{n+N}$,  we set
$$\eta=\eta_1\wedge\cdot\cdot\cdot\wedge\eta_N,~~~~\vartheta=\vartheta_1\wedge\cdot\cdot\cdot\wedge\vartheta_{n-N},$$
where
$$\eta_i=\sum_{s=1}^{n}\eta_{is}e_s+\sum_{l=1}^{N}\eta_{i,n+l}\varepsilon_l,\vartheta_j=\sum_{s=1}^{n}\vartheta_{js}e_s+\sum_{l=1}^{N}\vartheta_{j,n+l}\varepsilon_l,~~~~
i=1,\cdot\cdot\cdot,N,~j=1,\cdot\cdot\cdot,n-N.$$
Hence
\begin{align*}
&(\bigwedge_{n} d^{\mathcal{M}}\pi)(\eta\wedge\vartheta)=(\bigwedge_{N} d^{\mathcal{M}}\pi)\eta \wedge (\bigwedge_{n-N} d^{\mathcal{M}}\pi) \vartheta\\
&=\sum_{s=1}^{n}(\eta_1\cdot \nabla^{\mathcal{M}}\pi^s)e_s\wedge\cdot\cdot\cdot\wedge\sum_{s=1}^{n}(\eta_N\cdot \nabla^{\mathcal{M}}\pi^s)e_s\wedge \sum_{s=1}^{n}(\vartheta_1\cdot \nabla^{\mathcal{M}}\pi^s)e_s\wedge\cdot\cdot\cdot\wedge \sum_{s=1}^{n}(\vartheta_{n-N}\cdot \nabla^{\mathcal{M}}\pi^s)e_s\\
&=\sum_{s=1}^{n}\eta_{1s}e_s\wedge\cdot\cdot\cdot\wedge\sum_{s=1}^{n}\eta_{Ns}e_s\wedge \sum_{s=1}^{n}\vartheta_{1s}e_s \wedge\cdot\cdot\cdot\wedge\sum_{s=1}^{n}\vartheta_{n-N,s}e_s\\
&=\det A~ e_1\wedge\cdot\cdot\cdot\wedge e_{n},
\end{align*}
where $A=(a_{ij})$ is a $n\times n$ matrix with
$$a_{ij}=\begin{cases}
\eta_{ij}~~~~1\leq i\leq N,1\leq j\leq n  \\
\vartheta_{i-n,j}~~~~N<i\leq n, 1\leq j\leq n\\
\end{cases}.$$
Therefore $\det A=1$, which implies that $\sum_{s=1}^{n}\vartheta_{1s}e_s,\cdot\cdot\cdot,\sum_{s=1}^{n}\vartheta_{n-N,s}e_s$ are linear independent vectors in $\mathbb{R}^{n}$. Then
\begin{align*}
J^{\mathcal{M}_z}_{\pi}(x,y)&=\left|(\bigwedge_{n-N} (d^{\mathcal{M}}\pi) \vartheta\right|=\left|\sum_{s=1}^{n}\vartheta_{1s}e_s \wedge\cdot\cdot\cdot\wedge\sum_{s=1}^{n}\vartheta_{n-N,s}e_s\right|\\
&=(\sum_{\alpha\in I(n-N,n)}|M^{\overline{0}}_{\overline{\alpha}}(B)|^2)^{\frac{1}{2}}>0
\end{align*}
where $B=(\vartheta_{ij})$ is a $(n-N)\times n$ matrix. This completes the proof of the claim.

We choose a cut-off function $\chi_R\in  C_c^{\infty}(\mathbb{R}^N)$ such as (\ref{subequ401}), using coarea formula we have
\begin{equation}\label{subformula503}
\begin{split}
&G_u\llcorner (\widehat{\pi}^{\sharp}dy) (\omega(x) \chi_R(y))=\int_{\mathbb{R}^N}\left\{\int_{\mathcal{M}_{z}} \langle\omega(x)\chi_R(y),\vartheta(x,y) \rangle d\mathcal{H}^{n-N}\right\} d\mathcal{H}^N(z)\\
&=\int_{\mathbb{R}^N}\left\{\int_{\mathcal{M}_{z}^{+}} J^{\mathcal{M}_z}_{\pi}(x,y)\langle\omega(x)\chi_R(y),\frac{\bigwedge_{n-N} (d^{\mathcal{M}}\pi) \vartheta(x,y)}{|\bigwedge_{n-N} (d^{\mathcal{M}}\pi) \vartheta(x,y)|} \rangle d\mathcal{H}^{n-N}\right\} d\mathcal{H}^N(z)\\
&=\int_{\mathbb{R}^N}\left\{\int_{u^{-1}(z)\cap E_u} \langle\omega(x)\chi_R(u(x)),\frac{\bigwedge_{n-N} (d^{\mathcal{M}}\pi) \vartheta(x,u(x))}{|\bigwedge_{n-N} (d^{\mathcal{M}}\pi) \vartheta(x,u(x))|} \rangle d\mathcal{H}^{n-N}(x)\right\} d\mathcal{H}^N(z)
\end{split}
\end{equation}
By the dominated convergence theorem,
\begin{equation}\label{subfor530}
\pi_{\sharp}\left(G_u\llcorner (\widehat{\pi}^{\sharp}dy)\right) (\omega)=\int_{\mathbb{R}^N}\left\{\int_{u^{-1}(z)\cap E_u} \langle\omega(x),\frac{\bigwedge_{n-N} (d^{\mathcal{M}}\pi) \vartheta(x,u(x))}{|\bigwedge_{n-N} (d^{\mathcal{M}}\pi) \vartheta(x,u(x))|} \rangle d\mathcal{H}^{n-N}(x)\right\} d\mathcal{H}^N(z).
\end{equation}
This completes the proof of the theorem.
 \end{proof}

\begin{proof}[\bf Proof of Theorem \ref{subth502}]
Let $U\in W^{1,N}(\Omega\times (0,1))$ and $\widetilde{\omega}\in \mathcal{D}^{n-N}(\Omega\times [0,1))$ be the extension of $u$ and $\omega$, respectively.
According to proposition \ref{subpro511}, it follows that
$$[Ju](\omega)=(-1)^{(n-N)(N-1)+1}\pi_{\sharp}(G_U\llcorner \widehat{\pi}^{\sharp}dy)(d\widetilde{\omega}),$$
combing with Lemma \ref{sublem501}, we can easily obtain that
$$[Ju](\omega)=\int_{\mathbb{R}^N} T_yu(d\widetilde{\omega})dy.$$
where $T_yu=\tau (U^{-1}(y)\cap E_{U},1, \zeta),$ and $E_U:=R_U\cap \{\widetilde{x}\in \Omega\times(0,1)\mid J_{U}(\widetilde{x})>0\}$.
For any  $\eta \in \mathcal{D}^{n-N+1}(\Omega\times (0,1))$, it can be written as
$$\eta=\sum_{\alpha\in I(n-N+1,n+1)}\eta_{\alpha}(x,x_{n+1})dx^{\alpha},$$
where $\eta_{\alpha}\in C_c^{\infty}(\Omega\times (0,1))$. On the one hand, an argument similar to the one used in (\ref{subformula503}) and the differentiation theory show that
\begin{equation}\label{subeq5020}
(-1)^{(n-N)(N-1)+1}\frac{G_U\llcorner \widehat{\pi}^{\sharp}(\chi_{B(y,r)}dy)}{\omega_Nr^N}\rightharpoonup T_yu~~~~\mbox{in}~\mathcal{D}_{n-N+1}(\Omega\times (0,1)).
\end{equation}
for $\mathcal{H}^N$-a.e. $y\in \mathbb{R}^N$, where $\omega_N$ is the volume of the unit ball in $\mathbb{R}^N$. On the other hand
\begin{equation}\label{sub5009}
\begin{split}
&(-1)^{(n-N)(N-1)+1}G_U\llcorner \widehat{\pi}^{\sharp}(\chi_{B(y,r)}dy)(\eta)=(-1)^{(n-1)} G_U(\eta(\widetilde{x}) \chi_{B(y,r)}\wedge dy)\\
&=(-1)^{(n-1)}\sum_{\alpha\in I(n-N+1,n+1)}\sigma(\alpha,\overline{\alpha})\int_{\Omega\times (0,1)}\eta_{\alpha}(\widetilde{x}) \chi_{B(y,r)}(U(\widetilde{x})) M_{\overline{\alpha}}^{\overline{0}}(DU(\widetilde{x}))d\widetilde{x}\\
&=(-1)^{(n-1)}\int_{\Omega\times (0,1)}\left\langle \sum_{\alpha\in I(n-N+1,n+1)}\eta_{\alpha} \chi_{B(y,r)}(U)dx^{\alpha},\sum_{\alpha\in I(n-N+1,n+1)}\sigma(\alpha,\overline{\alpha})M_{\overline{\alpha}}^{\overline{0}}(DU) e_{\alpha}\right\rangle d\widetilde{x}\\
&=\int_{\Omega\times (0,1)} \langle \chi_{B(y,r)} (U)\eta, \varrho \rangle d\widetilde{x},
\end{split}
\end{equation}
where $\varrho(\widetilde{x})=\sum_{\alpha\in I(n-N+1,n+1)}(-1)^{(n-1)}\sigma(\alpha,\overline{\alpha})M_{\overline{\alpha}}^{\overline{0}}(DU)(\widetilde{x}) e_{\alpha}\in \bigwedge_{n-N+1}\mathds{R}^{n+1}$.
Let us denote by $J_U$ the Jacobian of $U$, thus Binet-Cauchy formula of determinant implies that
$$J_{U}=\sum_{\alpha\in I(n-N+1,n+1)}|M_{\overline{\alpha}}^{\overline{0}}(DU)|=|\varrho|.$$
According to  theorem 3 in  \cite[Vol. I, Sec. 3.1.4]{GMS},
there exist a non decreasing sequence of disjoint measurable set $\{F_k\}$ and Lipschitz functions $U_k$ such that
$$R_U\subset \cup_k F_k,~~~~U_k=U ~\mbox{on}~F_k,~~~~DU_k=DU~\mbox{a.e. on} ~F_k.$$
Then the coarea formula implies that
\begin{equation}\label{sub5008}
\begin{split}
&(-1)^{(n-N)(N-1)+1}\frac{G_U\llcorner \widehat{\pi}^{\sharp}(\chi_{B(y,r)}dy)}{\omega_Nr^N}(\eta)=\frac{1}{\omega_Nr^N}\int_{ E_U} \langle \chi_{B(y,r)}(U)\eta, \frac{\varrho}{J_{U}} \rangle J_{U}d\widetilde{x}\\
&=\frac{1}{\omega_Nr^N}\sum_{k=1}^{\infty}\int_{E_{U}\cap F_k} \langle \chi_{B(y,r)}(U)\eta, \frac{\varrho}{|\varrho|} \rangle J_{U_k}d\widetilde{x}\\
&=\frac{1}{\omega_Nr^N}\sum_{k=1}^{\infty}\int_{\mathbb{R}^N}\chi_{B(y,r)}\left\{\int_{U^{-1}(y)\cap E_{U}\cap F_k}\langle \eta, \frac{\varrho}{|\varrho|}\rangle d\mathcal{H}^{n-N+1}\right\}d\mathcal{H}^N(y)\\
&=\fint_{B(y,r)}\left\{\int_{U^{-1}(y)\cap E_{U}}\langle \eta, \frac{\varrho}{|\varrho|}\rangle d\mathcal{H}^{n-N+1}\right\}dy.
\end{split}
\end{equation}
Combing with (\ref{subeq5020}) we can easily prove (\ref{subth5021}). Then the proof is completed.
\end{proof}

\begin{remark}
Note that Hang-Lin's result in \cite{HL} show that $[J,u,y_1]=[J,u,y_2]$ for $\mathcal{H}^N$-a.e. $y_1,y_2\in B(0,1)$ if $u\in BNV(\Omega, S^{N-1})$.
\end{remark}
Next we shall establish the strong coarea formula for $BNV$.
For simplicity we divide Theorem \ref{subdetthm006} in the following two propositions.
\begin{proposition}[\bf Strong coarea formula I]
Let $u\in W^{1,N}(\Omega, \mathbb{R}^N)$  with $2\leq N\leq n$. Then
$$\left\|[Ju]\right\|(A)=\int_{\mathbb{R}^N} \left\| [J,u,y]\right\|(A)dy$$
for any Borel set $A\subset \Omega$.
\end{proposition}
\begin{proof}
For any Radon measure $\mu$ and any Borel set $A$
$$\mu(A)=\inf\{\mu(O)\mid O ~\mbox{open},~A\subset O\}.$$
So it suffices to prove the statement under the assumption that $A$ is open. Without loss of generality we can assume $A=\Omega$.
Fix $\omega\in \mathcal{D}^{n-N}(\Omega)$, combining  Proposition \ref{sub412} with \ref{subpro511} we can assert that
$$[Ju]=(-1)^{(n-N)N}\pi_{\sharp}(G_u\llcorner \widehat{\pi}^{\sharp}dy).$$
Let $\varphi$ be a real valued Borel function on $\mathbb{R}^N$, an argument similar to the one used in (\ref{subformula503}) shows that
$$G_u\llcorner (\widehat{\pi}^{\sharp}(\varphi dy)) (\omega(x))
=\int_{\mathbb{R}^N}\varphi(y)\left\{\int_{u^{-1}(y)\cap E_u} \langle\omega(x),\frac{\bigwedge_{n-N} (d^{\mathcal{M}}\pi) \vartheta(x,u(x))}{|\bigwedge_{n-N} (d^{\mathcal{M}}\pi) \vartheta(x,u(x))|} \rangle d\mathcal{H}^{n-N}(x)\right\} d\mathcal{H}^N(y).$$
It follows from the differentiation theory that for $\mathcal{H}^N$-a.e. $y\in \mathbb{R}^N$
\begin{equation}\label{subformula512}
(-1)^{(n-N)N}\frac{G_u\llcorner (\widehat{\pi}^{\sharp}(\chi_{B(y,r)} dy))}{\omega_Nr^N} \rightharpoonup \tau(u^{-1}(y)\cap E_u,1,\zeta_u)~~\mbox{in}~\mathcal{D}_{n-N}(\Omega),
\end{equation}
where $\omega_N$ is the volume of the unit ball in $\mathbb{R}^N$ and $\zeta_u=(-1)^{(n-N)N}\frac{\bigwedge_{n-N} (d^{\mathcal{M}}\pi) \vartheta(x,u(x))}{|\bigwedge_{n-N} (d^{\mathcal{M}}\pi) \vartheta(x,u(x))|}$.

Let $U\in W^{1,N}(\Omega\times (0,1))$ and $\widetilde{\omega}\in \mathcal{D}^{n-N}(\Omega\times [0,1))$ be the extension of $u$ and $\omega$, respectively.
In the same manner we can see that for $\mathcal{H}^N$-a.e. $y\in \mathbb{R}^N$
\begin{equation*}
(-1)^{(n-N)(N-1)+1}\frac{G_U\llcorner \widehat{\pi}^{\sharp}(\chi_{B(y,r)}dy)(d\widetilde{\omega})}{\omega_Nr^N}\rightarrow T_yu(d\widetilde{\omega}).
\end{equation*}
Combining with Theorem \ref{subth502}, proposition \ref{subpro511} and  (\ref{subformula512}),
$$[J,u,y]=\tau(u^{-1}(y)\cap E_u,1,\zeta_u)$$
for  $\mathcal{H}^N$-a.e. $y\in \mathbb{R}^N$. Note that $T_yu$ and $[J,u,y]$ are defined in Theorem \ref{subth502}. According to the coarea formula we have
\begin{align*}
\left\|[Ju]\right\|(\Omega)&=\int_{\Omega}\sum_{\alpha\in I(n-N,n)}|M^{\overline{0}}_{\overline{\alpha}}(Du(x))| dx=\int_{\mathbb{R}^N} \mathcal{H}^{n-N}(u^{-1}(y)\cap E_u)d\mathcal{H}^N(y)\\
&=\int_{\mathbb{R}^N}\|[J,u,y]\|(\Omega) d\mathcal{H}^N(y).
\end{align*}
The proof is completed.
\end{proof}

\begin{proposition}[\bf Strong coarea formula II]
Let $u\in BNV(\Omega, \mathbb{R}^N)$. Suppose that
\begin{enumerate}
\item[{\em(\romannumeral1)}]   $\int_{\mathbb{R}^N} \left\| [J,u,y]\right\|(\Omega)dy<\infty$
\item [{\em(\romannumeral2)}] Either $u\in C^0(\Omega,\mathbb{R}^N)$ or for any open set $V\subset \Omega$  there exist a sequence $\{u_k\}_{k=1}^{\infty}\subset BNV\cap C^{\infty}(\Omega,\mathbb{R}^N)$ such that
$$u_k\rightarrow u~\mbox{in}~W^{1-\frac{1}{N},N}~~~~\left\|[Ju_k]\right\|(V)\rightarrow \left\|[Ju]\right\|(V).$$
\end{enumerate}
Then
$$\left\|[Ju]\right\|(A)=\int_{\mathbb{R}^N} \left\| [J,u,y]\right\|(A)dy$$
for any Borel set $A\subset\subset \Omega$.
\end{proposition}
\begin{proof}
It suffices to prove the statement under the assumption that $A=\Omega$.
\begin{align*}
\left\|[Ju]\right\|(\Omega)&=\sup\{\int_{\mathbb{R}^N}[J,u,y](\omega) dy\mid \omega\in\mathcal{D}^{n-N}(\Omega),\|\omega\|\leq 1\}\leq \int_{\mathbb{R}^N} \left\|[J,u,y]\right\| (\Omega)dy.
\end{align*}
We now prove the other inequality. Let $U\in W^{1,N}(\Omega\times(0,1),\mathbb{R}^N)$ with $U|_{\Omega}=u$ and $S:=\left(\partial(G_U\llcorner \widehat{\pi}^{\sharp}dy)\right)\llcorner \Omega \times \mathbb{R}^N$.

\textbf{step 1:} First we will prove that
\begin{equation}
\|S\|(\Omega\times \mathbb{R}^N)\leq \left\|[Ju]\right\|(\Omega).
\end{equation}
If $u$ is continuous  in $\Omega$,
for any $\omega\in \mathcal{D}^{n-N,0}(\Omega\times \mathbb{R}^N)$ and its extension $\widetilde{\omega}\in \mathcal{D}^{n-N,0}(\Omega\times [0,1)\times \mathbb{R}^N)$.
$\widetilde{\omega}$ can be written as
$$\widetilde{\omega}=\sum_{\alpha\in I(n-N,n)}\widetilde{\omega}_{\alpha}(\widetilde{x},y)dx^{\alpha}.$$
It is clear that
\begin{align*}
S(\omega)&=\partial(G_U\llcorner \widehat{\pi}^{\sharp}dy) (\widetilde{\omega})=(-1)^{(n-N+1)N} G_U (d\widetilde{\omega}\wedge dy )\\
&=(-1)^{(n-N+1)N} \sum_{\alpha \in I(n-N,n)} \sum_{i\in \overline{\alpha}+(n+1)} \sigma(i,\alpha) G_U ( \partial_{x_i} \widetilde{\omega}_{\alpha} (\widetilde{x},y)dx^{\alpha+i} \wedge dy )\\
&= \sum_{\alpha \in I(n-N,n)} \sigma(\overline{\alpha},\alpha) \sum_{i\in \overline{\alpha}+(n+1)}-\sigma(\overline{\alpha}-i+(n+1),i)\int_{\Omega\times(0,1)} \partial_{x_i}\widetilde{\omega}_{\alpha}(\widetilde{x},U(\widetilde{x})) M^{\overline{0}}_{\overline{\alpha}-i+(n+1)}(DU(\widetilde{x}))d\widetilde{x}.
\end{align*}
Since
\begin{align*}
&\sum_{i\in \overline{\alpha}+(n+1)}-\sigma(\overline{\alpha}-i+(n+1),i)\sum_{j=1}^N \partial_{y_j}\widetilde{\omega}_{\alpha}(\widetilde{x},U(\widetilde{x})) D_iU^j(\widetilde{x})M^{\overline{0}}_{\overline{\alpha}-i+(n+1)}(DU(\widetilde{x}))\\
&=\sum_{j=1}^N (-1)^{N+1}\partial_{y_j}\widetilde{\omega}_{\alpha}(\widetilde{x},U(\widetilde{x})) \sum_{i\in \overline{\alpha}+(n+1)} D_iU^j(\widetilde{x})\sigma(i,\overline{\alpha}-i+(n+1))M^{\overline{0}}_{\overline{\alpha}-i+(n+1)}(DU(\widetilde{x}))\\
&=\sum_{j=1}^N -\partial_{y_j}\widetilde{\omega}_{\alpha}(\widetilde{x},U(\widetilde{x}))  M^{\overline{0}+j}_{\overline{\alpha}+(n+1)}(DU(\widetilde{x}))\\
&=0,
\end{align*}
Therefore
\begin{align*}
S(\omega)&=\sum_{\alpha \in I(n-N,n)} \sigma(\overline{\alpha},\alpha) \sum_{i\in \overline{\alpha}+(n+1)}-\sigma(\overline{\alpha}-i+(n+1),i)\int_{\Omega\times(0,1)} D_ i[\widetilde{\omega}_{\alpha}(\widetilde{x},U(\widetilde{x}))] M^{\overline{0}}_{\overline{\alpha}-i+(n+1)}(DU(\widetilde{x}))d\widetilde{x}\\
&=(-1)^{(n-N)N}[Ju](\widetilde{\omega}(\widetilde{x},U(\widetilde{x}))|_{\Omega})\\
&=(-1)^{(n-N)N}[Ju](\omega(x,u(x)))
\end{align*}
which implies that $\|S\|(\Omega\times \mathbb{R}^N)\leq \|[Ju]\|(\Omega)$. Note that as $\|[Ju]\|(\Omega)<\infty$, we can think of $[Ju]$ as being defined on the space of all $n-N$-forms with Borel bounded coefficients in $\Omega$.

If $u$ is not continuous, let $S_k:=\left(\partial(G_{U_k}\llcorner \widehat{\pi}^{\sharp}dy)\right)\llcorner \Omega \times \mathbb{R}^N$, it is clear that
$$S_k \rightharpoonup S ~\mbox{in}~ \mathcal{D}_{n-N}( \Omega \times \mathbb{R}^N).$$
By the lower semicontinuity of the mass
$$\|S\|(\Omega\times \mathbb{R}^N)\leq\liminf_{k \rightarrow \infty} \|S_k\|(\Omega\times \mathbb{R}^N) \leq \liminf_{k \rightarrow \infty}\|[Ju_k]\|(\Omega)=\|[Ju]\|(\Omega).$$

\textbf{step 2:}: Next we will show that $\|S\|(\Omega\times \mathbb{R}^N)\geq \int_{\mathbb{R}^N}\left\|[J,u,y]\right\|(\Omega)dy$.

Set $g(y)=\left\|[J,u,y]\right\|(\Omega)$.
It is clearly that $g\in L^{1}(\mathbb{R}^N)$ and hence  set
\begin{equation}
W:=\left\{z\in \mathbb{R}^N\mid \lim_{r\rightarrow \infty}\fint_{B(z,r)} g(y)dy=g(z)<\infty~~and~(\ref{subeq5020})~\mbox{holds} \right\}.
\end{equation}
Fix an $\varepsilon>0$, for any $z\in W$, there exsits $\omega\in \mathcal{D}^{n-N}(\Omega)$ such that
$$\|\omega\|\leq 1,~~~~ [J,u,z](\omega)\geq (1-\epsilon) \left\|[J,u,z]\right\|(\Omega)= (1-\epsilon)g(z).$$
(\ref{subeq5020}) and the differentiation theory implies that there exists $r_0(z,\omega)$ such that for any $r<r_0$
$$\fint_{B(z,r)} T_zu(d\widetilde{\omega}) dy\geq (1-\varepsilon) T_zu(d\widetilde{\omega}),~~~g(z)\geq \frac{1}{1+\varepsilon} \fint_{B(z,r)} g(y)dy.$$
Hence
\begin{equation}
\fint_{B(z,r)} [J,u,y](\omega) dy=\fint_{B(z,r)} T_yu(d\widetilde{\omega})  dy\geq (1-\varepsilon) [J,u,z](\omega).
\end{equation}
Moreover it is not difficult to find a function $\varphi\in C^{\infty}_c(B(z,r))$ such that $|\varphi|\leq 1$ and
$$(-1)^{(n-N)N}S(\varphi(y)\omega(x))=\int_{\mathbb{R}^N} \varphi(y)[J,u,y](\omega)dy\geq (1-\varepsilon) \int_{B(z,r)} [J,u,y](\omega) dy,$$
which implies that
$$(-1)^{(n-N)N}S(\varphi(y)\omega(x))\geq \frac{(1-\varepsilon)^3}{1+\varepsilon} \int_{B(z,r)} g(y) dy.$$
Define
\begin{align*}
\mathcal{F}&:=\{B(z,r)\in \mathbb{R}^N\mid z\in W, ~\mbox{there exists}~ \phi=\varphi(y)\omega(x)\in \mathcal{D}^{n-N}(\Omega \times B(z,r))
\\
 &\mbox{s.t.}~ \|\phi\|\leq 1, (-1)^{(n-N)N}S(\varphi(y)\omega(x)) \geq  \frac{(1-\varepsilon)^3}{1+\varepsilon} \int_{B(z,r)} \left\|[J,u,y]\right\|(\Omega) dy\}.
\end{align*}
According to the standard covering arguments, there exist a countable pairwise disjoint collection $\{B(z_k,r_k)\}_{k=1}^{\infty}$ and $\phi_k=\varphi_k(y)\omega_k(x)\in \mathcal{D}^{n-N}(\Omega \times B(z_k,r_k))$ such that
$$\mathcal{L}^N(\mathbb{R}^N\backslash \cup_{k=1}^{\infty}B(z_k,r_k))=0;~~~~\mbox{spt}(\phi_k)\cap \mbox{spt}(\phi_l)=\emptyset, ~k\neq l. $$
Therefore
$$\|S\|(\Omega\times \mathbb{R}^n)\geq \limsup_{p\rightarrow \infty}(-1)^{(n-N)N}S(\sum_{k=1}^p\phi_k(x,y)) \geq  \frac{(1-\varepsilon)^3}{1+\varepsilon} \int_{\mathbb{R}^n} \left\|[J,u,y]\right\|(\Omega) dy.$$
If we let $\varepsilon \rightarrow 0$ we have the desired inequality.
\end{proof}

Last we turn to the chain rule of $BNV(\Omega, \mathbb{R}^N)$.

\begin{proof}[\bf proof of Theorem 1.6]
According to the approximation theorem, there exists a sequence $\{u_k\}_{k=1}^{\infty}\subset W^{1-\frac{1}{N},N}(\Omega,\mathbb{R}^N)\cap C^{\infty}(\overline{\Omega},\mathbb{R}^N)$ such that $u_k\rightarrow u$ in $W^{1-\frac{1}{N},N}(\Omega,\mathbb{R}^N)$.
Let $\widetilde{u}, \widetilde{u}_k$ be the extension of $u, u_k$ to $\mathbb{R}^n$ in $W^{1-\frac{1}{N},N}$ and $U, U_k$ be the extension by average of $\widetilde{u}, \widetilde{u}_k$ to $\Omega\times [0,+\infty)$, i.e.,
$$U_k(x,x_{n+1})=\fint_{B(x,x_{n+1})} \widetilde{u}_k(z) dz,~~~k=0,1,2,\cdot\cdot\cdot,$$
where $B(x,x_{n+1})=\{x'\in \mathbb{R}^n\mid |x'-x|< x_{n+1}\}$. By the standard trace theory, $U,U_k\in W^{1,N}(\Omega\times (0,1))$ and
\begin{equation*}
\|U_k-U\|_{W^{1,N}}\leq C\|u-u_k\|_{W^{1-\frac{1}{N},N}}\rightarrow 0
\end{equation*}
as $k\rightarrow \infty$. Since $F\in C^1$ with $DF\in L^{\infty}$, we have $F(U_k)\rightarrow F(U)\in W^{1,N}(\Omega\times (0,1),R^N)$ and
$D(F(U))=DF(U)DU$ for $\mathcal{L}^{n+1}$-a.e. $\widetilde{x}\in \Omega\times(0,1)$.
 We denote $T$ be the linear operator from $W^{1,N}(\mathbb{R}^{n+1})$ to $W^{\frac{1}{N},N}(\mathbb{R}^{n})$, then
 \begin{equation}\label{sobolevlemma002}
 \|TF(U_k)-TF(U)\|_{W^{1-\frac{1}{N},N}}\leq C\|F(U_k)-F(U)\|_{W^{1,N}}\rightarrow 0
 \end{equation}
as $k\rightarrow \infty$.  Notice that $F$ and $U_k$ are continuous, then  $TF(U_k)=F(u_k)$, i.e. $F(U_k)|_{\Omega}=F(u_k)$. Combing with  (\ref{sobolevlemma002})  we can easily obtain that
$$F(U)|_{\Omega}=F(u).$$
Then for any $\omega\in D^{n-N}(\Omega)$ and $\widetilde{\omega}\in  D^{n-N}(\Omega\times [0,1))$ with $\widetilde{\omega}|_{\Omega}=\omega$,
$$[JF(u)](\omega)=\pi_{\sharp}(G_{F(U)}\llcorner \widehat{\pi}^{\sharp}dy)(d\widetilde{\omega}).$$
For simplicity of notation, we write $\eta$ instead of $d\widetilde{\omega}\in \mathcal{D}^{n-N+1}(\Omega\times [0,1))$. So it can be written as
$$\eta=\sum_{\alpha\in I(n-N+1,n+1)}\eta_{\alpha}(x,x_{n+1})dx^{\alpha},$$
where $\eta_{\alpha}\in C_c^{\infty}(\Omega\times [0,1))$. An argument similar to the one used in (\ref{sub5009}) and (\ref{sub5008}) shows that
\begin{align*}
[JF(u)](\omega)&=(-1)^{(n-1)}\sum_{\alpha\in I(n-N+1,n+1)}\sigma(\alpha,\overline{\alpha})\int_{\Omega\times (0,1)}\eta_{\alpha}(\widetilde{x}) M_{\overline{\alpha}}^{\overline{0}}(D[F(U)(\widetilde{x})])d\widetilde{x}\\
&=(-1)^{(n-1)}\int_{\Omega\times (0,1)}\det DF(U(\widetilde{x}))\langle \eta,\sum_{\alpha\in I(n-N+1,n+1)}\sigma(\alpha,\overline{\alpha})M_{\overline{\alpha}}^{\overline{0}}(DU) e_{\alpha}\rangle d\widetilde{x}\\
&=\int_{\mathbb{R}^N} \left\{\int_{U^{-1}(y)\cap E_U} \det DF(U(\widetilde{x})) \langle \eta, \zeta\rangle\right\}dy\\
&=\int_{\mathbb{R}^N} \det D F(y) [J,u,y](\omega)dy
\end{align*}
\end{proof}

As a consequence of the Theorem \ref{subdetthm003}, we have the following structure result.
\begin{theorem}
Let $u\in BNV(\Omega, \mathbb{R}^N)$ and $T_u$ be given in Definition \ref{subdef42}.  If $T_u$ has finite mass, then $u\in BV(\Omega, \mathbb{R}^N)$ and
$$
([Ju]^{\alpha})^{\mbox{ac}}(x)=\sigma(\alpha,\overline{\alpha})M_{\overline{\alpha}}^{\overline{0}}( ap Du(x)), ~~\mbox{for a.e.}~x\in \Omega,
$$
 for any $\alpha \in I(n-N,n)$. Where  the signed Radon measure of $[Ju]^{\alpha}$ is defined by  $[Ju]^{\alpha}(\psi):=[Ju](\psi dx^{\alpha})$ for any $\psi \in C_c^{\infty}(\Omega)$.
\end{theorem}
%---------------------------------------------------------------------------------------%

\section*{Acknowledgments}
\addcontentsline{toc}{chapter}{Acknowledgements}
This work is supported by NSF grant of China ( No. 11131005, No. 11301400) and Hubei Key Laboratory of Applied Mathematics (Hubei University).

\bibliographystyle{plain}

\begin{thebibliography}{1}

\bibitem{AM}
E. Acerbi and D. Mucci,
\emph{Graphs of maps between manifolds in trace spaces and with vanishing mean oscillation},
Ann. Inst. H. Poincaré Anal. Non Linéaire, \textbf{31} (2014),  1015-1034.

\bibitem{LND}
L. Ambrosio, N. Fusco and D. Pallara, \emph{Functions of bounded variation and
  free discontinuity problems}, vol. 254, Clarendon Press Oxford, 2000.

\bibitem{AK}
L. Ambrosio and B. Kirchheim, \emph{Currents on metric spaces},
Acta Math., \textbf{185} (2000), 1-80.


\bibitem{BJ}
J. Ball,  \emph{Convexity conditions and existence theorems in nonlinear elasticity},
Arch. Ration. Mech. Anal., \textbf{63} (1977), 337-403.

\bibitem{BN}
H. Brezis and H. Nguyen, \emph{The Jacobian determinant revisited,}
Invent. Math.,  \textbf{185} (2011),  17-54.

\bibitem{ER}
L. Evans and R. Gariepy, \emph{Measure theory and fine properties of
  functions}, vol. 5, CRC press, 1991.

\bibitem{FH}
H. Federer, \emph{Geometric measure theory}, vol. 1996, Springer New York,
  1969.

\bibitem{GE}
E. Gagliardo, \emph{Caratterizzazione delle tracce sulla frontiera relative ad alcune classi di funzioni in n variabili},
Rend. Semin. Mat. Univ. Padova,  \textbf{27} (1957), 284-305.

\bibitem{GMS2}
M. Giaquinta, G. Modica and J. Sou\v{c}ek,
\emph{Graphs of finite mass which cannot be approximated in area by smooth graphs},
Manuscripta Math,  \textbf{78} (1993), 259-271.


\bibitem{GMS}
M. Giaquinta, G. Modica and J. Sou\v{c}ek, \emph{Cartesian currents in the
  calculus of variations, I, II,}
Springer-Verlag, Berlin, 1998.

\bibitem{HL}
 F. Hang and F. Lin, \emph{A remark on the Jacobians,}
   Commun. Contemp. Math., \textbf{2} (2000), 35-46.

\bibitem{JS}
R. Jerrard and M. Soner, \emph{Functions of bounded higher variation,}
Indiana Univ. Math. J.,  \textbf{51}  (2002), 645-677.

\bibitem{LE1}
C. Lellis, \emph{Some fine properties of currents and applications to distributional Jacobians,}
Proc. Roy. Soc. Edinburgh Sect. A,  \textbf{132} (2002), 815-842.

\bibitem{LE2}
C. Lellis, \emph{Some remarks on the distributional Jacobian,}
Nonlinear Anal., \textbf{53} (2003), 1101-1114.

\bibitem{MC}
C. Morrey,  \emph{Multiple Integrals in the Calculus of Variations. Die Grundlehren der mathematischen Wissenschaften}, vol. 130, Springer, New York(1966).
%\bibitem{MIR}
%M. Miranda, \emph{Superfici cartesiane generalizzate ed insiemi di perimetro
 % localmente finito sui prodotti cartesiani}, Annali della Scuola Normale
 % Superiore di Pisa-Classe di Scienze \textbf{18} (1964) 515-542.

\bibitem{RY}
Y. Reshetnyak, \emph{The weak convergence of completely additive vector-valued set functions},
Sib. Mat. Zh.,  \textbf{9} (1968), 1386-1394.



\bibitem{SIM}
L. Simon, \emph{Lectures on Geometric Measure Theory},
Porc. of the Center for Mathematical Analysis, Australian National University, Vol. 3 (1983).

\bibitem{STE1}
E. Stein, \emph{The characterization of functions arising as potentials I},
Bull. Amer. Math. Soc., \textbf{67} (1961), 102-104.


\bibitem{STE2}
E. Stein, \emph{The characterization of functions arising as potentials II},
Bull. Amer. Math. Soc., \textbf{68} (1962), 577-582.

\end{thebibliography}

\end{document}